\documentclass[12pt, twoside]{article}
\usepackage{amsmath,amsthm,amssymb}
\usepackage{times}

\def\titlerunning#1{\gdef\titrun{#1}}
\makeatletter
\def\author#1{\gdef\autrun{\def\and{\unskip, }#1}\gdef\@author{#1}}
\def\address#1{{\def\and{\\\hspace*{18pt}}\renewcommand{\thefootnote}{}%
		\footnote {#1}}%
	\markboth{\autrun}{\titrun}}
\makeatother

\def\subjclass#1{{\renewcommand{\thefootnote}{}%
		\footnote{\emph{Mathematics Subject Classification (2010):} #1}}}
\def\keywords#1{\par\medskip
	\noindent\textbf{Keywords.} #1}

\numberwithin{equation}{section}

\usepackage{ifthen,bbm,fixmath}
\usepackage{graphicx}
\usepackage[shortlabels]{enumitem}

\usepackage{hyperref}
\hypersetup{colorlinks=true,citecolor=blue,filecolor=blue,linkcolor=blue,urlcolor=blue}

\newcommand{\R}{{\mathbb R}}
\newcommand{\Z}{{\mathbb Z}}

\newcommand{\N}{{\mathbb N}}

\newcommand{\sphere}{{\mathbb S}}
\newcommand{\ga}{\gamma}
\newcommand{\Ga}{\Gamma}
\newcommand{\La}{\Lambda}
\newcommand{\Om}{\Omega}
\newcommand{\de}{\delta}
\newcommand{\be}{\beta}

\newcommand{\si}{\sigma}
\newcommand{\eps}{\varepsilon}
\newcommand{\iso}{\operatorname{Iso}}
\newcommand{\SO}{\operatorname{SO}}

\newcommand{\supp}{\operatorname{Supp}}

\newcommand{\actson}{\curvearrowright}

\renewcommand{\cal}[1]{{\mathcal #1}}
\newcommand{\C}[1]{{\protect\cal #1}}

\newcommand{\I}[1]{{\mathbbm #1}}
\renewcommand{\O}[1]{\overline{#1}}
\newcommand{\G}[1]{\mathfrak{#1}}
\newcommand{\V}[1]{\mathbold{#1}}

\newcommand{\e}{\varepsilon}

\renewcommand{\mid}{\colon}
\newcommand{\dd}{\,\mathrm{d}}
\newcommand{\Ad}{\mathrm{Ad}}
\newcommand{\Id}{\operatorname{Id}}

\newif\ifnotesw\noteswtrue

\newcommand{\hide}[1]{}

\def\mathclap#1{\text{\hbox to 0pt{\hss$\mathsurround=0pt#1$\hss}}}

\newcommand{\beq}{\begin{equation}}
\newcommand{\eeq}{\end{equation}}

\newtheorem{theorem}{Theorem}[section]

\newtheorem*{theorem*}{Theorem}

\newtheorem{lemma}[theorem]{Lemma}

\newtheorem{claim}{Claim}[theorem]
\newtheorem{prop}[theorem]{Proposition}

\newtheorem{corollary}[theorem]{Corollary}

\theoremstyle{definition}
\newtheorem{definition}[theorem]{Definition}
\newtheorem{example}[theorem]{Example}

\newtheorem{remark}[theorem]{Remark}
\newtheorem{remarks}[theorem]{Remarks}

\newtheorem{assumptions}[theorem]{Assumption}

\newcommand{\bpf}[1][Proof.]{\smallskip\noindent{\it #1} }
\renewcommand{\qed}{\nolinebreak\mbox{\hspace{5 true pt}%
  \rule[-0.85 true pt]{3.9 true pt}{8.1 true pt}}}
\newcommand{\epf}{\qed \medskip}

\renewcommand{\ge}{\geqslant}
\renewcommand{\le}{\leqslant}
\renewcommand{\preceq}{\preccurlyeq}

\renewcommand{\subset}{\subseteq}
\renewcommand{\supset}{\supseteq}
\renewcommand{\ldots}{\hspace{0.9pt}.\hspace{0.3pt}.\hspace{0.3pt}.\hspace{1.5pt}}
\newcommand{\SL}{\operatorname{SL}}

\newcommand{\Matrix}[2]{\left[\begin{array}{#1} #2\end{array}\right]}
\newcommand{\CU}{\C C}

\usepackage[margin=1cm]{caption}

\begin{document}
	
	
\baselineskip=17pt

\title{Measurable equidecompositions for group actions with an expansion property}
\titlerunning{Measurable equidecompositions}

\author{{\L}ukasz Grabowski
	\and 
	Andr\'as M\'ath\'e
\and
Oleg Pikhurko}

\date{}

\maketitle

\address{F1. Grabowski: Department of Mathematics and Statistics,
	Lancaster University,
	Lancaster LA1 4YF, UK
	\and
	F2. M\'ath\'e: Mathematics Institute, University of Warwick, Coventry CV4 7AL, UK
\and
F3. Pikhurko: Mathematics Institute and DIMAP, University of Warwick, Coventry CV4 7AL, UK}



\begin{abstract}
Given an action of a group $\Gamma$ on a measure space $\Omega$, we provide a sufficient criterion under which two sets $A, B\subset \Omega$ are \textit{measurably equidecomposable}, i.e.,~$A$ can be partitioned into finitely many measurable  pieces which can be rearranged using some elements of $\Gamma$ to form a partition of $B$. In particular, we prove that every bounded measurable
subset of $\R^n$, $n\ge 3$, with non-empty interior is measurably equidecomposable to a ball via  isometries. The analogous result also holds for some other spaces, such as the sphere or the hyperbolic space of dimension $n\ge 2$.

\end{abstract}


\section{Introduction}

In this paper, we present a general sufficient criterion for equidecomposing a given pair of sets using measurable pieces. 
In order to state quickly some concrete results obtainable by our method, we discuss first the Euclidean space $\I R^n$,  which is probably the most important special case.  

\subsection{Euclidean space $\I R^n$}\label{se:IntroRn}

Let us call two subsets $A$ and $B$ of $\R^n$  \textit{(set-theoretically) equidecomposable} if it is possible to find a partition of $A$ into finitely many pieces and rearrange these pieces using  orientation-preserving isometries to form a partition of $B$. 
The most famous result about equidecomposable sets is known as the \textit{Banach-Tarski  paradox}: in 
$\R^3$, the unit ball and two disjoint copies of the unit ball are equidecomposable. It is a special case of the following theorem.

\begin{theorem}[Banach and Tarski \cite{BanachTarski24}]\label{th:BT} 
When $n\ge 3$, any two bounded subsets of $\I R^n$ with non-empty interiors are equidecomposable.\qed
\end{theorem}

An earlier result of Banach~\cite{Banach23} gives that, when $n\le 2$, equidecomposable subsets of $\I R^n$ which are measurable have the same Lebesgue measure.
In view of this, Tarski \cite{Tarski25} formulated the following problem, known as \textit{Tarski's circle squaring}: \textit{is a disk in $\R^2$ equidecomposable to a square of the same area?} 
Some 65 years later, Laczkovich \cite{Laczkovich90} showed that Tarski's circle squaring is possible.

There are various results which imply the impossibility of equidecompositions when additional regularity of the pieces is required. Examples include Dehn's theorem \cite{Dehn01} solving Hilbert's third problem and the result of Dubins, Hirsch and Karush \cite{DubinsHirschKarush63} which shows that circle squaring is not possible with Jordan domains. 

On the other hand, until recently there have been very few general positive results on the existence of \emph{measurable equidecompositions} (where each piece has to be Lebesgue measurable), although a related problem of measurable equidecompositions with countably many pieces was studied already by Banach and Tarski \cite[Th\'eor\`eme~42]{BanachTarski24}.
(See \cite[Section~11.3]{TomkowiczWagon:btp} for a survey of ``countable equidecompositions''.) 

Some progress has been recently made for sets with ``small boundary''. Namely,  if the boundaries of $A,B\subseteq \R^n$, $n\ge 1$, have upper Minkowski dimension strictly less than~$n$ while the sets have the same positive Lebesgue measure, then $A$ and $B$ are equidecomposable with pieces that are both Lebesgue and Baire measurable (Grabowski, M\'ath\'e and Pikhurko~\cite{GrabowskiMathePikhurko17}), Jordan measurable (M\'ath\'e, Noel and Pikhurko~\cite{MatheNoelPikhurko}), or Borel if the sets $A$ and $B$ are also Borel (Marks and Unger~\cite{MarksUnger17}).

In this paper we give a general criterion for measurable equidecomposability, which in particular applies to $\I R^n$ for  $n\ge 3$. An important feature of the present work when compared with \cite{GrabowskiMathePikhurko17,MarksUnger17,MatheNoelPikhurko} is 
that for a large natural class of sets $A\subset \R^n$, $n\ge 3$, we are able to completely characterise sets $B$ which are measurably equidecomposable to $A$. Furthermore, we do not need to assume anything about the boundaries of the sets. Given $n$, let $\CU$ consist of all bounded sets $A$ in $\I R^n$ such that some (equivalently, every bounded) non-empty open set can be covered by finitely many sets obtained by applying orientation-preserving isometries to~$A$. The 
following theorem is a direct consequence of our more general Corollary~\ref{cr:easy}.


\begin{theorem}\label{th:MainRn}
Let $n\ge 3$ and let $A\subseteq \R^n$ belong to $\CU$ and be Lebesgue measurable.
Then a set $B\subset \R^n$ is measurably equidecomposable to $A$  if and only if $B$ belongs to $\CU$,  is Lebesgue measurable, and has the same Lebesgue measure as $A$.\qed
\end{theorem}

In particular, for $n\ge 3$, every two bounded measurable subsets of $\I R^n$ with non-empty interior and of the same Lebesgue measure are measurably equidecomposable.



\begin{remarks}\label{re:1}\rm\begin{enumerate}[(i),nosep, topsep=0pt,partopsep=0pt, itemsep=1mm, wide]
\item \label{it:Rem1}The assumption that $n\ge 3$ is needed in Theorem~\ref{th:MainRn}. For example, in $\I R^2$ (resp.\ $\I R^1$), Laczkovich in~\cite[Theorem~3]{Laczkovich03} (resp.~\cite[Theorem 3.3]{Laczkovich93}) constructed continuum many Jordan domains with boundaries differentiable everywhere (resp.\ bounded sets that are countable unions of intervals) that all have measure 1 but no two of these sets are set-theoretically equidecomposable.

\item \label{it:Rem2}
Theorem~\ref{th:MainRn} can be combined with the results of Dougherty and Foreman \cite{DoughertyForeman94} (or recent generalisations by Marks and Unger \cite{MarksUnger16}) to show that if the sets $A$ and $B$ in the theorem also have the property of Baire, then the obtained pieces can be additionally required to have the property of Baire, see Corollary~\ref{cr:Leb}\ref{it:BL}. 

\item\label{it:Rem3} Various questions were posed early in the 20th century regarding the axiomatic theory of the Lebesgue measure, see e.g.\ \cite[Chapter~11]{TomkowiczWagon:btp}.
 A remaining key open problem is whether, for $n\ge 3$, every \emph{mean} (that is, a finitely additive and isometry invariant function) 
$\kappa:\C A\to[0,\infty)$ 
where $\C A$ is the family  of all bounded Borel subsets of $\I R^n$,
is a constant multiple of the Lebesgue measure. 
In a breakthrough, Margulis~\cite{Margulis82} showed that the answer is in the affirmative if we take for $\C A$ the (larger) family of all bounded measurable sets. This resolved the famous Banach-Ruziewicz Problem whose origins can be traced to the 1904 monograph of Lebesgue~\cite{Lebesgue04} (see the discussion in~\cite[Page~267]{TomkowiczWagon:btp}). The special case $\Omega=\I R^n$ of our Theorem~\ref{th:mean} 
reduces the family $\C A$  in the result of Margulis to those sets that additionally have the property of Baire. In brief, the connection comes from the observation that an equidecomposition  between $A$ and $B$ with all pieces in $\C A$ gives a certificate that $\kappa(A)=\kappa(B)$.

\item 
An upper bound on the number of pieces needed in Theorem~\ref{th:MainRn} can be derived from our proof. For example,
our estimates indicate that one can measurably equidecompose a  ball into a cube in $\I R^3$ using at most $20^{10^{10}}$ pieces, see~Remark~\ref{re:sample2}.
On the other hand, we do not know any non-trivial lower bounds for the number of necessary pieces.

\end{enumerate}
\end{remarks}

\subsection{A general criterion for measurable equidecompositions}\label{subsec-results}

In this section, we present our general criterion for measurable equidecomposability. We start by noting the following assumption that applies to the whole paper.

\begin{assumptions}
\label{as:1} Assume that $\Ga$ is a group, $(\Om,\tau)$ is a Polish topological space, $\cal B$ is the Borel $\si$-algebra of $(\Om,\tau)$, and $a:\Ga\actson \Om$ is a (left) action by Borel automorphisms.
\end{assumptions}

Furthermore, whenever we mention a measure $\mu$ on $\Om$, we additionally assume that the following holds.

\begin{assumptions}\label{as:mu} Assume that $\mu$ is a 
$\sigma$-finite measure on $(\Om,\cal B)$ which is non-zero on all non-empty open sets and that the action $a:\Ga\actson \Om$  preserves the measure~$\mu$.\end{assumptions}

We need a few definitions first. The result of the action of $\ga\,{\in}\, \Ga$ on $x\,{\in}\, \Om$ is denoted by $\ga.x:=a(\ga,x)$. Similarly, when $U\,{\subset}\, \Om$ we put $\ga.U := \{\ga.u\colon u\in U\}$, and when $T\subset \Ga$ we put $T.U:=\cup_{\ga\in T} \ga.U$. 
We say that a set $A\subset \Om$ \textit{covers} another set $B\subset\Om$ if there is a finite set $T\subset \Ga$ such that $B\subset T.A$. 

Let us also define various families of subsets of $\Omega$ that we will use. The completion of $\cal B$ with respect to $\mu$ is denoted by  $\cal B_\mu$. Elements of the $\sigma$-algebra $\cal B_\mu$ are called \textit{(Lebesgue) measurable}.  
Let $\C T$
denote the $\sigma$-algebra consisting of sets with the property of Baire. 
Let $\CU$  consist of all subsets of $\Om$ that have compact closure and cover a non-empty open set. While $\CU$ is closed under finite unions, it is not an algebra. Clearly, the $\sigma$-algebras $\C B$ and $\C B_\mu$ are invariant (under the action $a:\Ga\actson\Om$).  If, for example, $\Ga$ acts by homeomorphisms, then $\C T$ and $\CU$ are invariant too.

Given an $a$-invariant set family $\C A\subset2^\Omega$, two subsets $A,B\subseteq \Om$ are called \emph{$\C A$-equidecomposable (with respect to the action $a:\Ga\actson \Om$)} if for some $m\,{\in}\,\I N$ there exist group elements $\gamma_1,\ldots,\gamma_m\in\Gamma$ and a partition $A=A_1\sqcup\ldots\sqcup A_m$ with each piece belonging to $\C A$ such that $\ga_1.A_1,\ldots,\ga_m.A_m$ partition~$B$.
If $\C A$ is equal to $\C B$, $\C B_\mu$, $\C T$ and $\C B_\mu\cap\C T$, then an $\C A$-equidecomposition is called respectively 
\emph{Borel}, \emph{measurable}, \emph{Baire} and \emph{Baire-Lebesgue}. If $\C A=2^\Omega$, then we usually omit any reference to~$\C A$; 
however, if we need to emphasize that no restriction is imposed on the pieces, then we will use the term
\emph{set-theoretic equidecomposition}.

Given~$\mu$, we say that $A$ \textit{essentially covers} $B$ if $A$ covers $B\setminus N$ for some null set $N$.
If $A,B\subset \Om$ are measurable sets and there exist null sets $N$ and $N'$ such that  $A\setminus N$ and $B\setminus N'$  are Borel equidecomposable, then we say that $A$ and $B$ are \textit{essentially Borel equidecomposable}. It is easy to show (see Proposition~\ref{pr:combine}\ref{it:combine1}) that the existence of an essential Borel equidecomposition together with a set-theoretic equidecomposition is equivalent to the existence of a measurable equidecomposition.

We say that $C \in \cal B_\mu$ is a \textit{domain of expansion}  if  $0<\mu(C)<\infty$ and for every real $\eta>0$ there is a finite set  $R\subset \Gamma$ such that  \mbox{for all measurable
	sets $Y\subset C$} we have
$$
\mu ((R.Y) \cap C)\ge \min\left((1-\eta)\,\mu(C),\,\frac{\mu(Y)}{\eta}\right).
$$
Informally speaking, this states
	that all  measurable subsets of $C$ ``uniformly expand'' inside $C$ under a suitable finite subset $R$ of $\Gamma$, unless their $R$-images cover most of~$C$.
This property is crucial in the following general criterion, based on a result of Lyons and Nazarov~\cite{LyonsNazarov11}.

\begin{theorem}\label{th:suff} Let Assumptions \ref{as:1} and~\ref{as:mu} apply. Let  $A\subset \Om$ be a domain of expansion.
Then the following holds.
\begin{enumerate}[(i),nosep]
\item\label{it:suff1} A subset $B\subset\Om$ is essentially Borel equidecomposable to $A$ if and only if $A$ and $B$ essentially cover each other,  $B$ is measurable and $\mu(A)=\mu(B)$.
\item\label{it:suff2} A subset $B\subset\Om$ is measurably  equidecomposable to $A$ if and only if they are set-theoretically equidecomposable, $B$ is measurable and $\mu(A)=\mu(B)$.
\end{enumerate}
\end{theorem}

\subsection{Paradoxical actions}\label{se:paradox}

Equidecompositions have also been considered for spaces other than the Euclidean space $\I R^n$, often
with the aim of establishing ``paradoxes'' and concluding that certain kinds of measures do not exist. We refer the reader to
the excellent monograph on the subject by Tomkowicz and Wagon~\cite{TomkowiczWagon:btp}. 
Having a rich family of set-theoretic equidecompositions will be very useful
when applying Theorem~\ref{th:suff}\ref{it:suff2}. For the purposes of this paper,
we make the following (non-standard) definition.

\begin{definition}\label{de:paradox} Under Assumption~\ref{as:1}, the action $\Ga\actson\Om$ called \emph{paradoxical} if all the following properties hold.
	\begin{enumerate}[(i),nosep]
		\item \label{de:paradox1} The topological space $\Omega$ is locally compact.
		\item \label{de:paradox2} For every $\ga\in\Ga$ and every compact $C\subset\Om$,
		the closure of $\ga.C$ is compact.
		\item 
		\label{de:paradox3} 
		Any two subsets of $\Om$ with compact closure and non-empty interior are equidecomposable.
		\end{enumerate}
	\end{definition}

Of course, the crucial part of this definition is the last property.
The other (mild technical) properties will be needed in some of our arguments.
For example, Property~\ref{de:paradox2} implies that the family $\CU$ (which consists of sets with compact closure that cover a non-empty open set) is invariant. 
As a small detour, let us note the following  general proposition, whose
second part relies on  a powerful result of Marks and Unger~\cite{MarksUnger16}. 

\begin{prop}\label{pr:paradox} Let $a:\Ga\actson\Om$ be paradoxical
	and let $A\in\CU$.
	Then the following holds.
	\begin{enumerate}[(i),nosep]
		\item \label{it:paradox1}
		A subset $B\subset \Om$ is equidecomposable to $A$ if and only if $B\in\CU$. 
		\item\label{it:Paradox2} Suppose additionally that $A\in\C T$ and that the family of all meager subsets  of $\Om$ is $a$-invariant. Then a subset $B\subset \Om$ is $\C T$-equidecomposable to $A$ if and only if $B\in\CU\cap \C T$.
	\end{enumerate}
\end{prop}

Note that if a Borel bijection preserves meager sets, then it is also preserves sets with the property of Baire, so the family $\C T$ in Proposition~\ref{pr:paradox}\ref{it:Paradox2} is invariant.

Let us now look at some concrete examples of actions known to be paradoxical.
Let $\sphere^{n-1}$ denote the unit sphere in $\R^n$ with respect to the Euclidean metric. Let $\I H^n$ denote the $n$-dimensional hyperbolic space with the hyperbolic distance; see Section~\ref{se:Hn} for all formal definitions. 
Let $\iso(\R^n)$, $\SO(n)$ and $\iso(\I H^n)$ denote the group of orientation-preserving isometries of respectively $\R^n$, $\sphere^{n-1}$
and $\I H^n$. 
For each of these groups, we consider its natural action on the corresponding space.
Also, let $G_2$ be the subgroup of affine bijections of $\I R^2$ generated by $\SL(2,\I Z)$ (the linear maps given by $2\times2$ matrices with determinant 1 and all entries in $\I Z$) and all translations (that is, maps of the form $x\mapsto x+u$ for some vector $u\in\I R^2$).
This group naturally acts on $\I R^2$.

Hausdorff~\cite{Hausdorff14ma} showed that a ``third'' of $\I S^{n}$, $n\ge 2$, is equidecomposable to a ``half'' of $\I S^{n}$, which was enough to his intended application, namely, the non-existence of 
a mean defined on all subsets of the sphere. The paradoxicality of $\SO(n+1)\actson \I S^{n}$, $n\ge 2$, as stated
in Definition~\ref{de:paradox}\ref{de:paradox3}, was established by
Banach and Tarski \cite{BanachTarski24}. 
 The paradoxicality of $\iso(\I R^n)\actson\I R^n$ for $n\ge 3$ is the content of
Theorem~\ref{th:BT} (of Banach and Tarski~\cite{BanachTarski24}) while, as we mentioned already, this fails for $n\le 2$ by the results of Banach~\cite{Banach23}. One can ask
what happens if we allow a richer
 group of transformations of~$\I R^2$. The paradoxicality of $G_2\actson\I R^2$ was established in the influential paper of von Neumann~\cite{Neumann29} that introduced the concept of a non-amenable group.
 In fact, some (explicit) smaller subgroups of $G_2$ suffice here, see Mycielski~\cite[Corollary~5]{Mycielski98} (compare also with Wagon~\cite[Theorem~2]{Wagon82}). Laczkovich~\cite{Laczkovich99} 
 showed that the natural action $\SL(2,\I R)\actson\I R^2\setminus\{0\}$ is paradoxical. (See also Tomkowicz~\cite{Tomkowicz11} for a strengthening of this result.)
 	As noted by Mycielski~\cite[Page 143]{Mycielski89}, the paradoxicality of
$\iso(\I H^n)\actson \I H_n$, $n\ge 3$, can be established by observing that the subgroup of isometries that fix a point of $\I H^n$ acts on its every non-trivial orbit
in the same way as $\SO(n)$ acts on $\I S^{n-1}$. Mycielski~\cite{Mycielski89} showed that the action $\iso(\I H^2)\actson \I H_2$ is also paradoxical. (Some small gaps in Mycielski's proof were
fixed by Mycielski and Tomkowicz~\cite{MycielskiTomkowicz13}; the proof, with some further modifications, can also be found in~\cite[Theorem~4.17]{TomkowiczWagon:btp}.) 
We also refer the reader to Tomkowicz~\cite{Tomkowicz17} for a general short proof
of the paradoxicality of many of the above actions.

Let us collect these five (probably, best known) examples of paradoxical actions together. For later reference, we also state the standard invariant measure $\mu$ on $\Om$ in each case.\newpage

\begin{example}[Some known paradoxical actions $\Ga\actson\Om$]\label{ex:1}
	\mbox{}\nopagebreak
	\begin{enumerate}[(i),nosep]
		\item\label{it:Rn} $\Ga = \iso(\R^n)$ and $\Om = \R^n$ with the Lebesgue measure, $n\ge 3$, 
		\item\label{it:S} $\Ga = \SO(n)$ and $\Om=\sphere^{n-1}$ with the $(n-1)$-dimensional Hausdorff measure, $n\ge 3$, 
		\item\label{it:R2} $\Ga = G_2$ and $\Om =\R^{2}$ with the Lebesgue measure,
		\item\label{it:L99} $\Ga=\SL(2,\I R)$ and $\Om=\I R^2\setminus\{0\}$ with the Lebesgue measure,
		\item\label{it:Hn} $\Ga=\iso(\I H^n)$ and $\Om=\I H^{n}$ with the measure defined by~\eqref{eq:HnDensity}, $n\ge 2$.
	\end{enumerate} 
\end{example}

\subsection{Expanding actions}\label{se:appl}

Let us call an action $\Ga\actson\Om$ satisfying Assumptions~\ref{as:1} and~\ref{as:mu}  \emph{expanding (with respect to the measure $\mu$)} 
if every set in $\C B_\mu\cap \CU$ 
is a domain of expansion.

This notion is of interest because of the following corollary that can be derived with some work from Theorem~\ref{th:suff} and Proposition~\ref{pr:paradox}. 

\begin{corollary}\label{cr:Leb} In addition to Assumptions~\ref{as:1} and~\ref{as:mu}, assume that the action is paradoxical and expanding. Let $A\in \C B_{\mu}\cap \CU$. Then the following statements hold.
\begin{enumerate}[(i),nosep]
		
	\item\label{it:EssBor} A subset $B\subset \Om$ is essentially Borel equidecomposable to $A$  if and only if $B\in\C B_\mu$, $\mu(B)=\mu(A)$ and
	there is a null set $N$ with $B\bigtriangleup N\in \CU$.
	
\item\label{it:Mes} A subset $B\subset \Om$ is measurably equidecomposable to $A$  if and only if $B\in \C B_{\mu}\cap \CU$ and $\mu(B)=\mu(A)$.

\item\label{it:BL} Suppose additionally that $A$ has the property of Baire and
that the action preserves the family of meager sets. Then a subset $B\subset \Om$ is Baire-Lebesgue equidecomposable to $A$ if and only if $B\in\C B_{\mu}\cap \CU\cap \C T$ and $\mu(B)=\mu(A)$.
	\end{enumerate}
	
\end{corollary}

On the other hand, many natural actions can be shown to be expanding:

\begin{theorem}\label{th:DE} Each of the actions in Example~\ref{ex:1} is expanding.\end{theorem}

Each these actions is paradoxical and acts by homeomorphisms (in particular, preserving meager sets), so it satisfies all conclusions of Proposition~\ref{pr:paradox} and Corollary~\ref{cr:Leb}.
It is remarkable that one can obtain exact characterisations for so many types of equidecompositions under rather general assumptions (that in particular include all actions in Example~\ref{ex:1}). We did not see such a characterisation anywhere in the previous literature, surprisingly not even for set-theoretic equidecompositions in any case of Example~\ref{ex:1}. In fact, Tomkowicz and Wagon~\cite[Page~176]{TomkowiczWagon:btp} write: \emph{``It is not completely clear which subsets $E$ of $\I S^2$ are $\SO_3(\I R)$-equidecomposable with all of $\I S^2$.''} One answer to this question (exactly those sets that cover the whole of $\I S^2$) is given by Proposition~\ref{pr:paradox}\ref{it:paradox1}.

As an illustration, here is one (easy to state) direct consequence of Corollary~\ref{cr:Leb}\ref{it:Mes} and Theorem~\ref{th:DE}.

\begin{corollary}\label{cr:easy} For each action $\Ga\actson\Om$ of Example~\ref{ex:1}, every two measurable sets $A,B\in\CU$ of the same measure are equidecomposable with measurable pieces.\qed
\end{corollary}

\subsection{Connections to (local) spectral gap and finitely additive means}\label{intro:lsg}

It is not hard to show (see Proposition~\ref{pr:Spectral})
that  if $\mu$ is a finite measure and the action $\Ga\actson \Om$ has spectral gap, then the action is expanding.
Thus the case of $\SO(n)\actson \I S^{n-1}$, $n\ge 3$, of Theorem~\ref{th:DE} can be derived from the known spectral gap results established by  Drinfel'd \cite{Drinfeld84}, Margulis \cite{Margulis80}, and Sullivan \cite{Sullivan81}.

The case of the infinite measure space $\I R^n$ in Theorem~\ref{th:DE} is not directly covered by the above approach. However, we were able to derive it from the spectral gap of $\SO(3)\actson \I S^{2}$, using lengthy but rather elementary arguments. Later, we became aware of the general powerful results by Boutonnet, Ioana and Salehi Golsefidy~\cite{BoutonnetIoanaSalehi17} that can be used here.
Let us discuss this connection in general.

Let $X\subset \Om$ be a measurable set of finite positive measure. For $f\in L^2(\Om,\mu)$, we define  
$\|f\|_{2,X} := (\int_X |f(x)|^2\dd\mu(x))^{1/2}$. 
We say that an action $\Ga\actson \Om$ satisfying Assumptions~\ref{as:1} and~\ref{as:mu} has \textit{local spectral gap with respect to $X$} if there exist a finite set $Q\subset \Ga$ and a constant $c>0$  such that for any $f\in L^2(\Om,\mu)$ with $\int_X f(x) \dd\mu(x) =0$ we have
\beq\label{eq-local-gap}
\|f\|_{2,X}\le c\sum_{\ga\in  Q}\|\ga.f-f\|_{2,X},
\eeq
 where $\ga.f:\Om\to\I R$ is defined by $(\ga.f)(x):=f(\ga^{-1}.x)$, $x\in\Om$.

This notion is of interest to us because, as we will show in Lemma~\ref{lm:equivalence},
the action has local spectral gap with respect to a set $X\subset \Omega$ if and only if $X$ is a domain of expansion. 
Modulo Lemma~\ref{lm:equivalence}, Boutonnet et al~\cite[Theorem~A]{BoutonnetIoanaSalehi17} presented a sufficient condition for the action  to be expanding (stated here as Theorem~\ref{th:BIS:A}). While it seems that Theorem~\ref{th:BIS:A} can be used to derive the full Theorem~\ref{th:DE}, we use it here for $\SL(2,\I R)\actson \I R^2\setminus\{0\}$ and the hyperbolic space only,
presenting more direct proofs of the other cases of Theorem~\ref{th:DE}.

We decided to include also our initial proof of Theorem~\ref{th:DE} for $\iso(\I R^n)\actson \I R^n$. As we have already mentioned, it is rather elementary,  apart from using the spectral gap property of $\SO(3)\actson  \I S^2$. 
Also, it can be used to estimate the number of pieces in the obtained measurable equidecompositions; see Remark~\ref{re:sample2} for an example of a such calculation. 
Last but not least, our proof also gives the following result that does not seem to follow from~\cite{BoutonnetIoanaSalehi17} nor from other known spectral gap results.

\begin{theorem}\label{th:exotic} For each $n\ge 3$, there is a closed nowhere dense bounded subset $X$ of $\I R^n$ such that $\mu(X)>0$
and the action
$\iso(\I R^n)\actson \I R^n$ has local spectral gap with respect to~$X$ (i.e., by Lemma~\ref{lm:equivalence}, $X$ is a domain of expansion).
\end{theorem}

For an $a$-invariant family $\C A\subseteq 2^\Om$ which is closed under finite unions,
a \emph{mean} on $\C A$ is an $a$-invariant finitely additive function $\kappa:\C A\to[0,\infty)$. 
The analogue of the question discussed in Remark~\ref{re:1}\ref{it:Rem3},  namely whether every mean on Borel sets with compact closure is a constant multiple of the measure	$\mu$,
is also open for all actions listed in Example~\ref{ex:1} to the best of the authors' knowledge. 
The following theorem provides some partial progress in this direction.

\begin{theorem}\label{th:mean}
	Let $\Ga\actson\Om$ be any action from Example~\ref{ex:1} and let $\C A$ be the family of all measurable subsets of $\Om$ that have the property of Baire and compact closure. 
	Then every mean on $\C A$ is a constant multiple of the measure~$\mu$.\end{theorem}

The result of Margulis~\cite{Margulis82} (namely, the version of Theorem~\ref{th:mean} 
when we take the action $\iso(\I R^n)\actson \I R^n$ and enlarge the family $\C A$  by dropping the requirement that the sets in $\C A$ have the property of Baire) can be derived by a straightforward modification of our proof. Alternatively, it is a consequence of Theorem~\ref{th:DE}, Lemma~\ref{lm:equivalence} and the implication (4) $\Longrightarrow$ (1) of~\cite[Theorem~7.6]{BoutonnetIoanaSalehi17}.

\subsection{Organisation of the paper}

Section~\ref{aux} contain some further notation and various auxiliary results.

Theorem~\ref{th:suff} is proved in Section~\ref{sec-matchings-in-graphings} where we also give a full proof of the result of
Lyons and Nazarov~\cite{LyonsNazarov11} on the existence of a.e.-perfect Borel  matchings. 

Section~\ref{se:aux} is dedicated to proving Proposition~\ref{pr:paradox} and also contains the derivation of Corollary~\ref{cr:Leb}.

Section~\ref{sec-spectral-gap-and-expansion} discusses  the relation to (local) spectral gap in detail, in particular showing that the action has local spectral gap with respect to $X$ if and only if $X$ is a domain of expansion. 

Section~\ref{se:DE} contains the proofs of all cases of Theorem~\ref{th:DE}. In particular, Section~\ref{se:Rn} contains a few different proofs that the action $\iso(\I R^n)\actson \I R^n$ is expanding for $n\ge 3$, with our new (more elementary) proof appearing in Section~\ref{se:direct}.  This proof is then used to derive 
Theorems~\ref{th:exotic} (at the end of Section~\ref{se:direct}) and to give upper bounds on the number of pieces in some of our equidecompositions (in  Section~\ref{se:Numbers}). 

As we mentioned already, Corollary~\ref{cr:easy} clearly follows from Corollary~\ref{cr:Leb} and Theorem~\ref{th:DE}.

Theorem~\ref{th:mean} is proved in Section~\ref{se:mean}, as a consequence of a more general 
Lemma~\ref{lm:mean}.

As described in Section~\ref{se:IntroRn}, all new results stated there are direct consequences of the above results.

Some concluding remarks and remaining open questions can be found in Section~\ref{se:conluding}.

\section{Some further notation and auxiliary results}\label{aux}

Let us collect some frequently used notation, also recalling some definitions that already appeared in the Introduction.

Let $\I N:=\{0,1,\ldots\}$ consist of non-negative integers. For $k\in\I N$, we denote $[k]:=\{1,\ldots,k\}$. 
When we write $X=A\sqcup B$, we mean that the sets $A$ and $B$ \emph{partition} $X$ (i.e.,\ $A\cup B=X$ and $A\cap B=\emptyset$).  By $\pi_i$ we will denote the projection from a product to its $i$-th coordinate; formally, for sets $X_1,\ldots,X_m$ and $i\in [m]$, the projection $\pi_i$ maps $(x_1,\ldots,x_m)\in\prod_{j=1}^m X_j$ to~$x_i\in X_i$. 

Under Assumption~\ref{as:1}, we will use the following shorthands for $S,T\subseteq \Ga$, $\ga\in \Ga$, $y\in \Om$ and $X\subseteq \Om$:
\begin{eqnarray*}
	\ga.y&:=&a(\ga,y)\ \in\ \Om,\\
	\ga.X&:=&\{\ga.x\mid x\in X\}\ \subseteq\ \Om,\\
	S.X&:=&\cup_{\ga\in S}\,\ga.X\ =\ \{\ga.x\mid \ga\in S,\ x\in X\}\ \subseteq\ \Om,\\
	ST&:=&\{\sigma\gamma\mid \sigma\in S,\ \gamma\in T\}\ \subseteq\ \Ga.
	\end{eqnarray*}
We call $\ga.X$ a \emph{translate} of $X$.  The group action and the group multiplication take precedence over all other set operations; for example, $S.X\cap Y$ means $(S.X)\cap Y$.
The identity of the group $\Ga$ is denoted by~$e$. 
Also, we write
$S^{-1}:=\{\ga^{-1}\mid \ga\in S\}$ and call the set $S$ \emph{symmetric} if $S^{-1}=S$. 

By a \emph{multiset} $Q\subset \Ga$ we mean a function $Q:\Ga\to \I N$ with $Q(\ga)$ encoding the multiplicity of $\ga$ in $Q$. It is \emph{finite} if $|Q|:=\sum_{\ga\in\Ga} Q(\ga)$ is finite. For $f:\Ga\to\I R$, we define $\sum_{\ga\in Q} f(\ga):=\sum_{\ga\in\Ga} Q(\ga)\, f(\ga)$.

The closure of $X\subset \Om$ in the topological space $(\Om,\tau)$ is denoted by $\O X$. The \emph{complement} of $X$ is $X^c:=\Om\setminus X$. The \emph{indicator function} $\I 1_X$ of $X$ assumes value 1 on $X$ and value 0 on $\Om\setminus X$.

 Two sets $A,B\subset\Om$ \emph{cover each other}
if there is a finite set $S\subset \Ga$ with $S.A\supset B$ and $S.B\supset A$.
When we talk about equidecompositions for $\I R^n$, $\I S^{n}$, or $\I H^n$ without specifying the group, we mean by default the group of orientation-preserving isometries. 

Also, recall these families of subsets of $\Om$: $\C B$ (Borel), $\C B_\mu$ (measurable), $\CU$ (having compact closure and covering a non-empty open set) and $\C T$ (having the property of Baire).


\newcommand{\aee}{{a.e.}}

Let $G$ be a \emph{(bipartite) graph}, by which we mean a triple $(V_1,V_2,E)$, where $V_1$ and $V_2$ are sets (that are called the \emph{parts} of $G$) and $E$ (called the \emph{edge set} of $G$) is a subset of $V_1\times V_2$. 

A \emph{matching} in $G$ is a subset $M$ of the edge set $E$ such that for every $x$ in $V_1$ (resp.\ $V_2$) there is at most one vertex $y$ with $(x,y)\in M$ (resp.\ $(y,x)\in M$). 
For $i=1,2$,  the projection $\pi_i(M)$ consists of \emph{matched} vertices in $V_i$. It will be sometimes convenient to view a matching as a partial bijection; then $\pi_1(M)$ and $\pi_2(M)$ are just the domain and the range of~$M$.
The matching $M$ is called \emph{perfect} if $\pi_1(M)=V_1$ and $\pi_2(M)=V_2$, that is, $M$ as a function is a bijection from $V_1$ to~$V_2$.

For $X\subset V_1$ and $Y\subset V_2$ their
\emph{neighbourhoods} are respectively
\begin{eqnarray*}
N(X)&:=&\{y\in V_{2}\mid \exists\,x\in X\ (x,y)\in E\},\\ 
N(Y)&:=&\{x\in V_1\mid \exists\,y\in Y\ (x,y)\in E\}.
\end{eqnarray*}
 The \emph{degree} of a vertex $x$ is $\deg(x):=|N(\{x\})|$. We call $G$ \emph{locally finite}
 if the degree of each vertex is finite.

There may be some ambiguity when $V_1\cap V_2\not=\emptyset$. 
There are two, essentially equivalent, ways to deal with this formally.
The first one is to work  with the action $\tilde a:\Ga\oplus C_2\actson \Om\times C_2$ instead, where $C_2:=(\{-1,1\},\times)$ is the cyclic group with two elements,
 $\Ga$ acts on each copy of $\Om$ the same way as before while the non-identity element of $C_2$ swaps these two copies.
Then we can
replace $V_i$ by $V_i\times\{(-1)^i\}$, thus making the parts disjoint. However, then we have to make (routine) verifications  that our claimed results, when proved for~$\tilde a$, transfer to the original action~$a$.
Alternatively, we can operate with 
the \emph{unordered} graph 
\beq\label{eq:tilde}
 \tilde G:=((V_1\times\{-1\})\sqcup (V_2\times\{1\}),\tilde E),\quad \mbox{where }\tilde E:=\{\,\{(x,-1),(y,1)\}\mid (x,y)\in E\},
\eeq
which carries the same information as $G$. 
For example, instead of the degree of $x\in V_i$ we should have, strictly speaking, defined the degree of $(x,(-1)^i)$, etc.
Since the meaning will usually be clear from the context, we will be working mostly with~$G$, switching to $\tilde G$
occasionally.

We call a bipartite graph $G=(V_1,V_2,E)$ \emph{Borel} if $V_1,V_2\subseteq \Om$ and $E\subset \Om^2$ are all Borel. One way to generate Borel subsets of $\Om^2$ that is relevant to this paper is as follows.
A \textit{Borel arrow} is a pair $(U,\ga)$, where $U\subset \Om$ is Borel and  $\ga\in \Ga$.
Given a countable set of Borel arrows $\C A = \{(U_i,\ga_i)\mid i\in I\}$, let
 \beq\label{eq:EArrows}
 E(\C A):=\cup_{i\in I} \{(x,\ga_i.x)\mid x\in U_i\}\subseteq \Om^2.
 \eeq
 This set is Borel (as the countable union of the graphs of Borel partial functions). A special case is when $\C A=\{(\Omega,\ga)\mid \ga\in S\}$ for some countable $S\subset \Ga$ (that is, each $U_i$ is equal to $\Om$); then we denote
 \beq\label{eq:ES}
 E_S:=E(\C A)=\cup_{\ga\in S} \{(x,\ga.x)\mid x\in\Om\}\subset \Om^2,
 \eeq
 which is just the union  over $\ga\in S$ of the graphs of the functions $a(\ga,\cdot):\Om\to\Om$. 


We will be using the Lusin-Novikov Uniformisation Theorem (see e.g.~\cite[Theorem~18.10]{Kechris:cdst}) a number of times, often without explicitly mentioning it. We need only the weaker form of the theorem which states that, for Polish spaces, a continuous countably-to-one image of a Borel set is Borel. For example, one of its consequences is that for every Borel graph $G=(V_1,V_2,E)$ with countable vertex degrees and every Borel $X\subseteq V_i$ its neighbourhood $N(X)$
is Borel. Indeed, if e.g.\ $i=1$ then $N(X)=\pi_2((X\times V_{2})\cap E)$ is the countable-to-one image of a Borel set 
under the (continuous) projection~$\pi_2$. Another useful consequence of the Uniformisation Theorem is as follows.

\begin{lemma}\label{lm:Borel} Under Assumptions~\ref{as:1}, let $S\subset \Ga$ be a countable subset and let $E'\subset E_S$ be Borel. Then there are Borel sets $U_\ga\subset \Om$ for $\ga\in S$ such that
$\C A:=\{(U_\ga,\ga)\mid\ga\in S\}$ \emph{bijectively generates} $E'$ (meaning that $E'=E(\C A)$ and for every $(x,y)\in E'$ there is exactly one  $\ga\in S$  with $x\in U_\ga$ and $y=\ga.x$).
\end{lemma}
\begin{proof} In brief, we have to pick in a Borel way, for every edge of $E'$,  exactly one element of $S$ that generates this edge. Such a selection can be obtained by taking a maximal Borel independent set, which exists by a result of Kechris, Solecki and Todorcevic~\cite[Proposition~4.2]{KechrisSoleckiTodorcevic99}, in the (non-bipartite) Borel graph with vertex set $\{((x,y),\gamma) \in E'\times S\mid \gamma.x=y\}$ whose two vertices are connected if they correspond to the same edge of~$E'$. 

For completeness, we include a more direct proof.

Let $\preceq$ be a linear ordering of $S$ coming from some injection of $S$ into~$\I N$. We construct the required
sets $U_\ga$ one by one by taking the maximal possible set given the previous sets. 
Namely, we inductively define $U_\ga$ for $\ga\in S$ as
	$$
	U_\ga :=\{x\in \Om\mid (x,\ga.x)\in E'\}\setminus (\cup_{\beta\prec \ga} \{x\in U_\beta\mid \ga.x=\beta.x\}).
	$$
	The obtained sets can be shown to be Borel by induction on~$\ga\in S$. 
	For example, each auxiliary set 
	$$
	\{x\in \Om\mid \ga.x=\beta.x\}=\pi_1(\,\{(x,y)\in\Om^2\mid y=\gamma.x\}\cap \{(x,y)\in\Om^2\mid y=\beta.x\}\,)
	$$
	is Borel by the Lusin-Novikov Uniformisation Theorem.
	
	By definition, $E'\supset E(\C A)$, where $\C A:=\{(U_\ga,\ga)\mid \ga\in S\}$. Also, for every $(x,y)\in E'$ there is exactly one $\ga\in S$ with $x\in U_\ga$ and $y=\gamma.x$, namely, the smallest element of $\{\ga\in S\mid y=\ga.x\}$.
	(Note that this set is non-empty since $E'\subset E_S$.) It follows that $\C A$ has all required properties.
\end{proof}



The following result points out a well-known connection between equidecompositions and graph matchings.

\begin{lemma}\label{lm:MEQ} Under Assumption~\ref{as:1}, let $S\subseteq \Ga$ be a finite set and $M$ be a matching in the bipartite graph $(\Om,\Om,E_S)$. Let $A:=\pi_1(M)$ and $B:=\pi_2(M)$. 
	Then the following statements hold.
\begin{enumerate}[(i),nosep]
 \item\label{it:MEQ1} The sets $A,B\subset \Om$ are equidecomposable using group elements from~$S$ only. 
 \item\label{it:MEQ2} If  $M\subseteq \Om^2$ is Borel,
 then $A$ and $B$ are Borel equidecomposable  using group elements from~$S$ only.
 	\end{enumerate}
	\end{lemma}

\bpf 
For every $\ga\in S$, let $A_\ga$ consist of those $x\in A$ such that
$(x,\ga.x)\in M$ and $\ga$ is the smallest element of $S$ with this property (under a fixed ordering of~$S$). Since $M$ is a matching, the sets $A_\ga$ partition $A$ and their translates $\ga.A_\ga$ partition $B$. Thus the sets $A$ and $B$ are equidecomposable using elements from $S$ only. 

If, moreover, the matching $M$ is Borel, then all pieces $A_\ga$ are Borel (which can argued similarly as in the proof of Lemma~\ref{lm:Borel}).
\epf

Let us call the action $\Ga\actson\Om$ \emph{minimal} if there is no closed subset $X\subset \Omega$ such that $X$ is \emph{proper} (i.e., $X\not=\emptyset$ and $X\not=\Omega$) and $\Ga.X=X$. 

\begin{lemma}\label{lm:minimal} Under Assumption~\ref{as:1}, suppose that the action $\Ga\actson\Om$ is paradoxical. Then it is minimal. Also, if $A\in \CU$ is equidecomposable to some set $B$, then $B\in \CU$.
Also, any two sets $A,B\in\CU$ cover each other.\end{lemma}

\begin{proof} Suppose that a closed proper subset $X\subset \Om$ satisfies $\Ga.X= X$. Pick some
	$x\in X$ and $y\in \Om\setminus X$. By the local compactness of $\Om$, choose open sets $U\ni x$
	and $W\ni y$ with compact closures where, additionally, we can assume that $W\cap X=\emptyset$. Then $\Ga.W$, as a subset of the invariant set $\Om\setminus X$, does not contain $x\in U$. Thus the sets $U,W\in\CU$ are not
	equidecomposable, contradicting the paradoxicality of the action.

Next, suppose that $B\subset \Om$ is equidecomposable to $A\in\CU$. The set $B$ has to cover $A$ and, by the transitivity
of the covering relation, $B$ also covers some non-empty open set. On the other hand, $B$ is covered by finitely many
copies of $A$. Since the closure of $A$ is compact and the action of each element of $\Ga$ preserves this property by Definition~\ref{de:paradox}\ref{de:paradox2}, the closure of $B$ 
can be covered by finitely many compact sets and so it is compact itself. Thus $B\in\CU$ as required.

Finally, let $A,B\in\CU$ be arbitrary. By the definition of~$\CU$, the sets $A$ and $B$
cover some non-empty open sets $U$ and $W$  respectively. By shrinking $U$ and $W$, we can additionally assume that they have compact closures. By the paradoxicality of the action,
$A\cup U$ is equidecomposable to $B\cup W$, from which it easily follows that $A$ and $B$ cover each other.
\end{proof}

Let Assumption~\ref{as:mu} apply to the rest of this section. In particular, the action $\Ga\actson\Om$ preserves the measure $\mu$ on $(\Om,\C B)$. The following reformulation of the measure preservation property will be useful to us.



\begin{lemma}\label{lm:mp} Under Assumptions~\ref{as:1} and~\ref{as:mu}, 
let 
$\Lambda$ be a countable subgroup of $\Ga$,
let $C\subset \Om$ be Borel, and let $\psi:C \to\Om$ be a Borel injective map
	such that, for every $x\in C$, we have $\psi(x)\in \Lambda.\{x\}$.
Then 
	$\psi(C)$ is a Borel set and 
	$\mu(\psi(C))=\mu(C)$. 
\end{lemma}
\begin{proof} 
	Of course, the first claim (that $\psi(C)$ is Borel) follows from the Lusin-Novikov Uniformisation Theorem; we also get it as a by-product of our proof of the second claim.
		
	The graph $E':=\{(x,\psi(x))\mid x\in C\}$ of $\psi$ is a Borel subset of $E_\Lambda$. By Lemma~\ref{lm:Borel},
	we can find Borel arrows $(C_\beta,\beta)$, $\beta\in \Lambda$, that bijectively generate~$E'$. The sets $C_\beta$ partition~$C$ and their images $\psi(C_\beta)=\beta.C_\beta$ are Borel sets that partition~$\psi(C)$ by the injectivity of the function~$\psi$. 
	Thus $\psi(C)=\sqcup_{\be\in\Lambda} \be.C_\be$ is a Borel set
	of measure $\sum_{\beta\in\Lambda} \mu(\be.C_\be)=\sum_{\beta\in\Lambda} \mu(C_\be)=\mu(C)$, as required.
\end{proof}

We say that some property holds \emph{almost everywhere} (\emph{a.e.}\ for short) if the set of $x\in \Om$ where it fails is a null set with respect to the measure~$\mu$. For example, two sets $A,B\subset\Om$ \emph{essentially cover each other} if there is a finite $S\subset \Ga$ with $S.A\supset B$ and $S.B\supset A$ a.e. 

Recall that a measurable set $C \subset\Om$ is a \textit{domain of expansion}  if  $0<\mu(C)<\infty$ and for every real $\eta>0$ there is a finite set  $R\subset \Gamma$ such that for all measurable
	sets $Y\subset C$ we have
\beq\label{eq:DoE}
\mu (R.Y \cap C)\ge \min\left((1-\eta)\,\mu(C),\,\frac{\mu(Y)}{\eta}\right).
\eeq
Such a set $R$ will be called \textit{$\eta$-expanding for $C$}.
	Observe that it is enough to check~\eqref{eq:DoE} just for Borel subsets of~$C$. (Indeed, every $\sigma$-finite Borel measure on a Polish space is regular, see e.g.~\cite[Proposition 8.1.12]{Cohn13mt}, so for every measurable $Y$ there is a null set $N$ such that $Y\setminus N$ is Borel, in fact, an $F_\sigma$-set.) Also, $C'\subset \Om$ with $C'=C$ a.e.\ is a domain of expansion if and only if $C$ is.

\begin{lemma}\label{lm:Cover} Under Assumptions~\ref{as:1} and~\ref{as:mu}, let $A,B\subset \Om$ be two measurable sets that essentially cover each other. Then $A$ is a domain of expansion if and only if $B$ is a domain of expansion.
\end{lemma}
\begin{proof} 
	
	Let us assume, for example, that $A$ is a domain of expansion. Let $T\subset \Ga$ be a finite set such that $A\subset T.B$~and $B\subset T^{-1}.A$~a.e. Let $t:=|T|$. By the invariance of the measure, we have that
	$0<\mu(A)/t\le \mu(B)\le t\,\mu(A)<\infty$.

	Let $\e>0$ be arbitrary.  Let $\eta:=\e/t^2$ and%
	\hide{Fix sufficiently small $\eta>0$ so that
	$$
	\max\left(\eta\,\frac{t\mu(A)}{\mu(B)},\eta t^2\right)\le \e.
	$$}
	let $S$ be an $\eta$-expanding set for $A$. 
	
	We claim that $T^{-1}S T$ is an $\eps$-expanding set for $B$. Indeed, let $Y\subset B$ be a measurable set. We have that $Y\subset T^{-1}.A$ \aee, so there exists $\ga\in T$ such that  $\mu(\ga.Y\cap A)=\mu(Y\cap \ga^{-1}.A)\ge \mu(Y)/t$. Since $S$ is $\eta$-expanding for $A$, we have
	$$
	\mu(S T.Y\cap A) \ge \min\left((1-\eta)\,\mu(A),\,\frac{\mu(T.Y\cap A)}{\eta}\right)\ge \min\left((1-\eta)\,\mu(A),\,\frac{\mu(Y)}{\eta t}\right).
	$$
	If $\mu(A\setminus S T.Y) \le \eta\,\mu(A)$, then we have by the choice of $T$ that
	$$
	\mu(B\setminus T^{-1}S T. Y) \le t\, \mu(A\setminus S T.Y)\le t\,\eta\,\mu(A)= \frac{\e\, \mu(A)}{t}\le \e\,\mu(B).
	$$
	\hide{If $\mu(S T.Y\cap A) \ge (1-\eta)\,\mu(A)$ then, since $B\subset T^{-1}.A$ a.e., we have 
		$$
		\mu(T^{-1}S T. Y\cap B) \ge \mu(B) -t\,\eta\,\mu(A)= \left(1-\frac{t\eta\mu(A)}{\mu(B)}\right)\,\mu(B)\ge (1-\e)\,\mu(B),
		$$
		as desired.}%
	Otherwise, we have $\mu(S T.Y\cap A) \ge \mu(Y)/(\eta t)$. Using that $A\subset T.B$ a.e.,\ we deduce that for some $\ga\in T$ we have $
	\mu(S T.Y\cap \ga.B)\ge \mu(Y)/(\eta t^2)$, and so 
	$$
	\mu(T^{-1}S T.Y\cap B)\ge\mu(S T.Y\cap \ga.B) \ge \mu(Y)/(\eta t^2)= \mu(Y)/\e.
	$$ Thus, $T^{-1}ST$ is $\e$-expanding for $B$, as desired.\end{proof}

\begin{remark}\label{re:DoE}\rm Recall that we defined the action $\Ga\actson\Om$ to be \emph{expanding} if every $C\in\C B_\mu\cap\CU$ is a domain of expansion. It follows from Lemmas~\ref{lm:minimal} and~\ref{lm:Cover} that, for paradoxical actions,
	 it is enough to check
	 just one arbitrary member of $\C B_\mu\cap\CU$.\end{remark}

We call a bipartite Borel graph $G=(A,B,E)$ a \emph{bipartite $c$-expander} if $0<\mu(A)=\mu(B)<\infty$ and, for every measurable (equivalently, Borel) subset $Y$ of $A$ or $B$, it holds that
%
\beq\label{eq-hall-condition}
\mu(N(Y))>\frac12\,\mu(A)\quad\mbox{or}\quad\mu(N(Y)) \ge (1+c)\,\mu(Y).
\eeq
 We call $G$ a \emph{bipartite expander} if it is a \emph{bipartite $c$-expander} for some $c>0$.
For example, if $C\subset \Om$ is Borel with $0<\mu(C)<\infty$ and $S\subset\Ga$ is an $\eta$-expanding set for $C$ with $0<\eta<1/2$, then the graph $(C,C,E_S\cap C^2)$ is a bipartite $((1-\eta)/\eta)$-expander.

\section{Proof of Theorem \ref{th:suff}}

\label{sec-matchings-in-graphings}

\subsection{Augmenting paths}

First, we need to define some graph-theoretic concepts adopted to our setting. 
Let $G=(A,B,E)$ be a bipartite graph and let $M\subset E$ be a matching in~$G$.

An \emph{alternating path (starting at $A$)} is a non-empty sequence of points $P=(x_0,\ldots,x_\ell)$ such that
\begin{enumerate}[(i), nosep]
\item $x_0\in A\setminus \pi_1(M)$;
\item \label{it:odd}
 the odd-indexed vertices $x_1,x_3,\ldots\in B$ are  pairwise distinct;
\item 
for every $i\in [\ell]$,  we have that $(x_{i-1},x_i)\in E\setminus M$ if $i$ is odd and $(x_i,x_{i-1})\in M$ if $i$ is even.
\end{enumerate} 

Since $M$ is a matching and $x_0\in A$ is unmatched, Item~\ref{it:odd} implies that the even-indexed vertices $x_0,x_2,\ldots\in A$ are also pairwise distinct. 
 This definition works also when $A\cap B\not=\emptyset$, since the parity of the position in $P$ determines the part a vertex is assigned to. If we are to work with the unordered graph $\tilde G$ as defined in~\eqref{eq:tilde}, then the corresponding definition is that 
$$\tilde P:=((x_0,-1),(x_1,1),\ldots,(x_\ell,(-1)^{\ell+1}))$$ is a path in $\tilde G$ that starts with an unmatched vertex 
and whose edges alternate between $\tilde{E}\setminus\tilde{M}$ and $\tilde{M}$. The \emph{length} of $P$ is $\ell$, the number of edges.

An \emph{augmenting path (starting at $A$)} is an alternating path $P=(x_0,\ldots,x_\ell)$ of odd length $\ell\ge 1$ such that $x_\ell\in B\setminus \pi_2(M)$.
The \emph{augmentation} of $M$ along $P$
is the matching $M':=M\bigtriangleup E(P)$ which is obtained by taking the symmetric difference between $M$ and $$
E(P):=\{(x_0,x_1),(x_2,x_1),(x_2,x_3),\ldots,(x_{\ell-1},x_\ell)\}.
$$
In order words, we modify $M$ by including $(x_{i-1},x_i)$ for all odd $i\in [\ell]$ and removing $(x_i,x_{i-1})$ for all even $i\in [\ell]$. Note that the augmented matching $M'$ covers two new vertices: $\pi_1(M')=\pi_1(M)\sqcup\{x_0\}$ and $\pi_2(M')=\pi_2(M)\sqcup\{x_\ell\}$.

Suppose that $E\subset E_R$ for some fixed countable set $R\subset \Ga$, where $E_R$ is defined by~\eqref{eq:ES}. Then a \textit{Borel augmenting family} is
a tuple $(U,\beta_1,\ldots,\beta_\ell)$, where $U$ is a Borel subset of $A$ and
$\beta_1,\ldots,\beta_\ell$ are elements of $R$ (with repetitions allowed) so that (i) for every $x\in U$ the sequence 
 \begin{equation}\label{eq:AugFamPath}
 P_{x,\be_1,\ldots,\be_\ell}:=(x,\, \be_1.x,\, \be_2\be_1.x,\, \ldots, \, \be_\ell\ldots\be_2\be_1.x)
 \end{equation}
is an augmenting path for $M$; (ii) for distinct $x,y\in U$ the corresponding augmenting paths $\tilde P_{x,\be_1,\ldots,\be_\ell}$ and $\tilde P_{y,\be_1,\ldots,\be_\ell}$ in $\tilde G$ are vertex-disjoint. Informally speaking, a Borel augmenting family is a Borel collection of vertex-disjoint augmenting paths in~$\tilde G$. The \textit{length} of such a family is $\ell$, the number of edges in each path.
The \emph{augmentation} $M'$ of $M$ along a Borel augmenting family $(U,\beta_1,\ldots,\beta_\ell)$ is obtained from $M$ by augmenting along the paths in~\eqref{eq:AugFamPath} for all $x\in U$. (Since the paths are vertex-disjoint in $\tilde G$, all these augmentations can be done in parallel.) Note that $M'$ is a matching, $\pi_1(M') = \pi_1(M)\sqcup  U$ and $\pi_2(M')=\pi_2(M') \sqcup (\beta_\ell\ldots\beta_1.U)$. Also, $M'$ is Borel if $M$ is.

By swapping the roles of $A$ and $B$, we can define in the obvious way an \emph{alternating path starting at $B$}. Note that there is no need to define an augmenting path starting at $B$: by reversing the order of vertices we can consider augmenting paths starting at~$A$ only.

\subsection{Proof of the result of Lyons and Nazarov on measurable matchings}\label{subsec-matchings}

In this section we provide all details of the result of Lyons and Nazarov~\cite[Remark 2.6]{LyonsNazarov11} (stated as Theorem~\ref{th:LN} here) which gives a  sufficient condition for the existence of a Borel a.e.-perfect matching. We need the following result first, which does not use measure and, in fact, was proved in~\cite{ElekLippner10} for arbitrary (not necessarily bipartite) Borel graphs.

\begin{lemma}[Elek and Lippner~\cite{ElekLippner10}]\label{lm:EL}
Under Assumption~\ref{as:1}, let $G=(A,B,E)$ be a Borel locally finite bipartite graph with $E\subset E_R$
for some countable $R\subset \Ga$. (In particular, $A$ and $B$ are Borel subsets of $\Om$.)
Then there are Borel matchings
$M_0=\emptyset$, $M_1,M_2,\ldots  \subset E$
such that
\begin{enumerate}[(i),nosep]
\item\label{it:EL1} for every $i\ge 1$, the matching $M_i$ admits no augmenting path of length at most $2i-1$; 
\item\label{it:EL2} for every $i\ge 0$, there is a countable sequence of Borel matchings $K_0,K_1,\ldots\subset E$ such that (a) $K_0=M_{i}$; (b) for every $j\ge 1$, $K_j$ is the augmentation of $K_{j-1}$ along some Borel augmenting family of length at most $2i+1$; (c) every $(x,y)\in E$ eventually belongs to either all or none of the matchings $K_j$; (d) $M_{i+1}=\cup_{j=1}^\infty\cap_{t=j}^\infty K_t$. 
 \end{enumerate}
\end{lemma}

\begin{proof} The first item is \cite[Proposition 1.1]{ElekLippner10}. The second item follows directly from the proof of \cite[Proposition 1.1]{ElekLippner10}. 
In brief, the proof, when adopted to our setting, proceeds as follows. 

Suppose that $i\ge 0$ and we have already constructed matchings $M_0,\ldots,M_{i}$. By a result of Kechris et al~\cite{KechrisSoleckiTodorcevic99}, there is a Borel vertex colouring $\phi:A\to \I N$  such that $\phi(x)\not=\phi(y)$ for every two distinct vertices $x,y\in A$ with $(x,-1),(y,-1)\in \tilde A$ being at distance at most $4i+2$ in~$\tilde G$. Fix some sequence $(\V x_j)_{j=1}^\infty$ where each $\V x_j$ belongs to $X:=\cup_{t=1}^{i+1} R^{2t-1}\times \I N$ so that each element of $X$ appears as $\V x_j$ for infinitely many choices of $j$. 

Define $K_0:=M_{i}$. Suppose that $j\ge 1$ and we have already defined $K_{j-1}$. Let $\V x_j=((\beta_1,\ldots,\beta_{2t-1}),c)$ with $t\in[i+1]$. Let $U$ consist of those $x\in A$ such that $P_x:=P_{x,\beta_1,\ldots,\beta_{2t-1}}$, as defined in~\eqref{eq:AugFamPath}, is an augmenting path for $K_{j-1}$ and $\phi(x)=c$. The set $U$ is Borel. Also, for distinct $x,y\in U$, the paths $\tilde P_x$ and $\tilde P_y$ are vertex disjoint for otherwise 
$(x,-1)$ and $(y,-1)$ are at distance at most $4i+2$ in $\tilde G$ and satisfy $\phi(x)=\phi(y)=c$, contradicting the choice of the colouring $\phi$. 
 Thus $(U,\beta_1,\dots,\beta_{2t-1})$ is a Borel augmenting family for~$K_{j-1}$. Define $K_j$ to be the augmentation of $K_{j-1}$ along $(U,\be_1,\ldots,\be_{2t-1})$. Increase $j$ and repeat.

Each augmentation as above that flips any given edge $(x,y)\in E$ strictly decreases the set of currently unmatched vertices in $\tilde G$ at distance at most $2i$ from the edge. This set is finite, e.g.\ by K\"onig's Infinity Lemma. Thus $(x,y)$ can be flipped only finitely many times, that is, it eventually belongs to all or none of the matchings $K_j$. It follows that $M_{i+1}:=\cup_{j=1}^\infty\cap_{t=j}^\infty K_t$ is a Borel matching that does not admit any augmenting path of length at most $2i+1$ as every potential path was considered for augmentation for infinitely many values of $j$.\end{proof}

\begin{remark}\rm A result of Hopcroft and Karp~\cite{HopcroftKarp73} states that if a matching $K$ admits no augmenting path of length at most $2i-1$ and we flip an augmenting path of length $2i+1$ (that is, shortest possible), then no new augmenting paths of length at most $2i+1$ appear. Thus, when we construct $M_{i+1}$ from $M_{i}$ in the proof of Lemma~\ref{lm:EL}, it is enough in fact to take for $X$ just some enumeration of $R^{2i+1}\times \I N$.\end{remark}

\begin{theorem}[Lyons and Nazarov \cite{LyonsNazarov11}]
\label{th:LN} In addition to the assumptions of Lemma~\ref{lm:EL}, let Assumption~\ref{as:mu} apply and let the graph $G=(A,B,E)$ be a bipartite expander, that is, $\mu(A)=\mu(B)\in (0,\infty)$ and~\eqref{eq-hall-condition} holds for some $c>0$.
Then $G$ has a Borel a.e.-perfect matching $M\subset E$.
\end{theorem}

\bpf
Let $M_0,M_1,\ldots$ be the sequence of Borel matchings returned by Lemma~\ref{lm:EL}. Let $M$ be defined by
\begin{equation}\label{eq:M}
M:=\cup_{j=0}^\infty \cap_{i=j}^\infty M_i,
\end{equation}
that is, $M$ consists of edges that are eventually included in every $M_i$. Clearly, $M$ is a Borel matching. Thus it remains to show that the set of vertices not matched by $M$ has measure 0. As we will see, this will be a consequence of Claims~\ref{claim-bound} and~\ref{cl:muXi} via a Borel-Cantelli-type argument.

For $i\in \I N$, let $X^i := A\setminus \pi_1(M_i)$ and $Y^i := B\setminus \pi_2(M_i)$ be the subsets of $A$ and $B$ of vertices not matched by $M_i$. 
By Lemma~\ref{lm:mp}, the sets $\pi_1(M)$ and $\pi_2(M)$ (and thus the sets $X^i$ and $Y^i$) have the same measure. 

Note that the set $\pi_1(M_{i}\bigtriangleup M_{i+1})$ consists of those $x\in A$ on which $M_{i}$ and $M_{i+1}$ differ as partial functions (in particular, it includes those $x\in A$ for which exactly one of these functions is defined).

\begin{claim}\label{claim-bound} For every $i\in\I N$, it holds that $
 \mu(\pi_1(M_{i}\bigtriangleup M_{i+1}))\le (i+1)\,\mu(X^i).
 $
\end{claim}
\vspace{5pt}\noindent\textit{Proof of Claim.} In brief, when we construct $M_{i+1}$ from $M_i$, each individual augmentation that matches an extra point of $A$ changes the current matching at at most $i+1$ elements of $A$ so the claim follows from the invariance of the measure~$\mu$.

Let us give a full proof. Given $i$, let $K_0=M_i, K_1, K_2, \ldots\subset E$ be the 
Borel matchings returned by the second part of Lemma~\ref{lm:EL}. 


Take $j\ge 1$ and let $(U_j,\gamma_1,\ldots,\gamma_{2\ell+1})$ with $0\le \ell\le i$ be the Borel augmenting family
for $K_{j-1}$ that augments it into $K_{j}$. Since each augmentation increases the set of matched vertices and thus $\pi_1(K_{j-1})\supset \pi_1(M_i)$, we have that $U_j\subset A\setminus \pi_1(K_{j-1})$ is a subset of
$A\setminus\pi_1(M_i)= X^i$. For $t\in\{0,\ldots,\ell\}$,
let $\psi_t$ map $x\in U_j$ to $\ga_{2t}\ldots\ga_1.x$.
Then $\psi_0,\ldots,\psi_\ell$ are Borel injections 
 whose images 
partition $\pi_1(K_{j}\bigtriangleup K_{j-1})$. By Lemma \ref{lm:mp}, we have
$$
\mu (\pi_1(K_{j-1} \bigtriangleup K_{j}))  = \sum_{t=0}^\ell \mu(\psi_t(U_j))
=(\ell+1) \mu(U_j)\le  (i+1)\,\mu(U_j).
$$

Again, note that augmentations can only increase the set of matched vertices. Thus, for every $h>j$, we have that $U_h\subseteq A\setminus \pi_1(K_{h-1})$ is disjoint from
$U_j\subseteq \pi_1(K_j)\subseteq \pi_1(K_{h-1})$. Also, $M_{i}\bigtriangleup M_{i+1}\subseteq \cup_{j=1}^\infty (K_j \bigtriangleup K_{j-i})$. Finally, since every $U_j$ is a subset of $X^i$, we conclude that
$$
\mu(\pi_1(M_{i}\bigtriangleup M_{i+1}))\le \sum_{j=1}^\infty \mu (\pi_1(K_{j-1} \bigtriangleup K_{j}))  \le (i+1) \sum_{j=1}^\infty \mu(U_j)\le  (i+1)\,\mu(X^i).
$$
 The claim is proved.\qed

Fix $i\ge 1$. 
For integer $j\ge 0$, let $X_j^i$ consist of the end-points of alternating paths for $M_i$ (that start at $A$) whose length is at most $j$ and has the same parity as~$j$. Since $i$ is fixed (until~\eqref{eq:UnfixI}), we abbreviate $X_j:=X_j^i$.
Since, for each $x\in X^i$, the length-0 path $(x)$ is alternating, we have in particular that $X_0=X^i$. 
As the graph is bipartite, it holds that $X_0\subset X_2\subset\ldots\subset A$ and $X_1\subset X_3\subset\ldots \subset B$. Define $X_0':=X_0$, $X_1':=X_1$ and, for $j\ge 2$, $X_j':=X_{j}\setminus X_{j-2}$. By definition, the sets
$X_0',X_2',\ldots$ (resp.\ $X_1',X_3',\ldots$) are pairwise disjoint.

\begin{claim}\label{cl:bijection} For every $j\in [i]$, the matching $M_i$ gives a bijection between
$X_{2j}'$ and $X_{2j-1}'$.\end{claim}
\begin{proof}[Proof of Claim]
Let $x$ be any element of  $X_{2j}'\subset A$. By definition, there is an alternating path $(x_0,\ldots,x_{2j})$ with $x_0\in X_0$ and $x_{2j}=x$. This path has even length, so  $(x,y)\in M_i$, where $y:=x_{2j-1}$. The truncated alternating path $(x_0,\ldots,x_{2j-1})$ shows that $y\in X_{2j-1}$. Suppose, on the contrary to the claim, that $y\not\in X_{2j-1}'$. Then $y\in X_{2j-3}$. Let this be witnessed by some alternating path~$Q$ of length at most $2j-3$. The odd-length path $Q$ ends with an edge in~${E}\setminus {M_i}$. Either $Q$ already contains $x$ or we can extend $Q$ by adding $x$, in each case obtaining a contradiction to $x\not\in X_{2j-2}$.

Conversely, take any $y$ in $X_{2j-1}'\subset B$. Fix an alternating path $P=(x_0,\ldots,x_{2j-1})$ with $x_0\in X_0$ and $x_{2j-1}=y$. The last edge of this path of odd length is in ${E}\setminus {M_i}$ and, since $M_i$ admits no augmenting path of length $2j-1\le 2i-1$, the vertex $y$ has to be matched. Let $(x,y)\in M_i$. The vertex $x\in \pi_1(M_i)$ cannot belong to $X_0=A\setminus \pi_1(M_i)$. Also, $x$ cannot belong to $X_{2t}'$ for some $t\in[j-1]$: otherwise an even-length alternating path $Q$ witnessing this has to end with the pair $(x,y)\in M_i$ (as $y$ is the unique $M_i$-match of $x$) and a truncation of $Q$ shows that $y\in X_{2t-1}\subset X_{2j-3}$, a contradiction.
In particular, $x$ cannot occur in $P$ as an even-indexed vertex. Thus we can add $x$ to $P$, obtaining an alternating path which shows that $x\in X_{2j}$. We have already argued that $x\not\in X_{2j-2}$.
Thus we conclude that $x\in X_{2j}'$, as desired.\end{proof}

Recall that $X_j$ is a subset of $A$ (resp.\ $B$) if $j$ is even (resp.\ odd).

\begin{claim}\label{cl:N} For every $j\in [i]$, we have $X_{2j+1} = N(X_{2j})$.
\end{claim}
\begin{proof}[Proof of Claim.]
The inclusion $X_{2j+1}\subset N(X_{2j})$ is clear. For the other direction, take arbitrary $x\in X_{2j}$ and $y\in N(\{x\})\subset B$. Pick an alternating path $P$ of length at most $2j$ starting at $X_0$ and ending in $x$. Suppose that $y$ does not belong to $P$, as otherwise a truncation of $P$ shows that $y$ is in $X_{2j-1}\subset X_{2j+1}$, giving the required. In particular, we have that $(x,y)\not \in M_i$ as otherwise $y$ precedes $x$ on $P$. Thus we can extend $P$ by adding $y$. This shows that $y\in X_{2j+1}$, as desired.\end{proof}

Claims~\ref{cl:bijection} and~\ref{cl:N} show by induction on $j\in[i]$ that the set $X_j'$ is Borel. (In fact, this is true for every $j$ since $X_j'$ is defined by a local rule; however the stated range of $j$ will suffice for our purposes.)

\begin{claim}\label{claim-growth} For every $j\in [i]$ we have that
$\mu(X_{2j})\ge \mu(X_{2j-1})$ while $\mu(X_{2j-1})$
is strictly larger than $\frac12\,\mu(B)$ or at least $(1+{c})\,\mu(X_{2j-2})$.
\end{claim}
\begin{proof}[Proof of Claim] 
Claim~\ref{cl:bijection} implies by induction on $j\in[i]$ that the matching $M_i$ gives a bijection between $X_{2j}\setminus X_0=X_2'\sqcup \ldots\sqcup X_{2j}'$ and $X_{2j-1}=X_1'\sqcup\ldots\sqcup X_{2j-1}'$. 
Thus we have by Lemma~\ref{lm:mp} that $\mu(X_{2j}\setminus X_0) = \mu(X_{2j-1})$,
giving the first inequality.

The stated lower bound on $\mu(X_{2j-1})$ is a direct consequence of Claim~\ref{cl:N} and the bipartite expansion property of $G$ assumed by the theorem.
\end{proof}

For $j\in\I N$, define $Y_j$ to consist of the end-points of alternating paths which start in $Y_0:=Y^i$ whose length is at most $j$ and has the same parity as~$j$. (This becomes the same definition as that of $X_j$ when we swap the roles of the sets $A$ and $B$.)
By symmetry, Claim~\ref{claim-growth} also holds when we swap $A$ and $B$, and replace each $X_t$ by $Y_t$.


\begin{claim}\label{cl:Xi-1Yi} The sets $X_{i-1}$ and $Y_{i}$ are disjoint.
\end{claim}
\begin{proof}[Proof of Claim] 
Assuming the contrary, pick alternating paths $P=(x_0,\ldots,x_\ell)$ and $Q=(y_0,\ldots,y_k)$ starting at $A$ and $B$ respectively 
(thus $x_0\in X_0$ and $y_0\in Y_0$) such that $\ell+k$ is odd, $x_\ell=y_k$, $\ell \le i-1$, $k\le i$, and the value of $\ell+k$ is smallest possible. Note that $k\ge 1$ as otherwise the path $P$ of length $\ell\le 2i-1$ is $M_i$-augmenting, contradicting the choice of~$M_i$. Similarly, $\ell\ge 1$.

 The minimality of $\ell+k$ implies that the paths $\tilde P$ and $\tilde Q$ in $\tilde G$ intersect only in their end-points. 
Also, if $\ell$ is odd (resp.\ even), then exactly one of the edges $(x_{\ell-1},x_\ell),(y_{k-1},y_k)\in E$ 
(resp.\ $(x_{\ell},x_{\ell-1}),(y_{k},y_{k-1})\in E$) belongs to $M_i$; namely  $(y_{k-1},y_k)\in M_i$ (resp.\ $(x_\ell,x_{\ell-1})\in M_i$). Therefore the concatenation of $P$ with the reversal of $Q$ at the common end-point, that is, 
$
(x_0,\ldots,x_\ell,y_{k-1},\ldots,y_0)
$,
 is an augmenting path of length at most $2i-1$, contradicting the choice of $M_i$.\end{proof}

Note that $\mu(X^i) = \mu(Y^i)$, since $\mu(A)=\mu(B)<\infty$ by our assumption and
$\mu(\pi_1(M_i))=\mu(\pi_2(M_i))$ by Lemma~\ref{lm:mp}.

\begin{claim}\label{cl:muXi} It holds that
$
\mu(X^i)=\mu(Y^i) \le (1+c)^{\frac{-i+1}2}\mu(A).
$
\end{claim}
\begin{proof}[Proof of Claim]
Suppose first that $i$ is odd. By Claim~\ref{cl:Xi-1Yi}, the sets $X_{i-1}$ and $Y_i$ are disjoint. As both are subsets of $A$, at least one of these two sets, suppose $X_{i-1}$, has measure at most $\mu(A)/2$.  
By Claim~\ref{claim-growth}, the measures of the sets $X_0\subseteq X_2\subseteq\ldots\subseteq X_{i-1}$ increase by factor at least $1+c$ each time. Thus,
by $X_0=X^i$, we have
 $$
 \mu(A)
\ge \mu(X_{i-1}) \ge (1+c)^{\frac{i-1}2}\mu(X_0) = (1+c)^{\frac{i-1}2}\mu(X^i),
$$
 giving the required. The obvious modifications of this proof (swapping the roles of $A$ and $B$ if needed) also apply to the remaining cases (when $\mu(Y_i)\le \mu(A)/2$ or when $i$ is even), giving the stated bound.
\end{proof}

Now we are ready to finish the proof of the theorem (essentially by applying the Borel-Cantelli Lemma). Indeed, $\pi_1(M)$ contains every vertex of $A$ which matched by $M_i$ and on which all matchings $M_j$ with $j\ge i$ agree. Thus, for every $i\in \I N$, we have
 \begin{equation}\label{eq:UnfixI}
\mu(A\setminus \pi_1(M))\le \mu(X^i)+\sum_{j=i}^\infty \mu(\pi_1(M_{j+1}\triangle M_j)).
\end{equation}
 The last series is summable by Claims~\ref{claim-bound} and~\ref{cl:muXi}. Furthermore, Claim~\ref{cl:muXi} also gives that $\mu(X^i)$ tends to $0$. Since we can pick an arbitrarily large $i$, we have that
 $\mu(A\setminus \pi_1(M))=0$,  which finishes the proof of Theorem~\ref{th:LN}.
\qed

\subsection{Finishing the proof of Theorem \ref{th:suff}}


\begin{proof}[Proof of Theorem~\ref{th:suff}\ref{it:suff1}.] 
The only non-trivial part is to show the converse direction, namely that if $A$ is a domain of expansion (in particular, $0<\mu(A)<\infty$),
$A$ and $B$ essentially cover each other, and
$\mu(B)=\mu(A)$, then $A$ is essentially Borel equidecomposable to~$B$. By removing null sets, we can assume that $A$ and $B$ are Borel.

Choose a finite symmetric set $T\subset \Ga$ such that $T.A\supset B$ and $T.B\supset A$~a.e. By 
Lemma \ref{lm:Cover}, $B$ is also a domain of expansion. Let $t:=|T|$. Fix $\eta>0$ with $2t\eta<1$.
Take a finite subset $R\subset \Gamma$ which is $\eta$-expanding for both $A$ and $B$.

Let us we verify that the Borel graph $G:=(A,B,E_{TR\cup (TR)^{-1}}\cap (A\times B))$ satisfies the bipartite expansion condition \eqref{eq-hall-condition} with $c:=1$.  Let $Y$ be a measurable subset of, say, $A$. Since $R$ is $\eta$-expanding
for $A$, at least one of the following two alternatives hold.
If $\mu(R.Y\cap A)\ge \mu(Y)/\eta$ then, by $T.B\supset A\supset R.Y\cap A$ a.e., 
there is $\ga\in T$ with the translate $\ga.B$ covering at least $1/t$-th measure of $R.Y\cap A$; thus
 \begin{eqnarray*}
 \mu (N(Y))&\ge& \mu(TR.Y\cap B)\ \ge\ \mu(\ga^{-1}.(R.Y)\cap B)\\
 &=& \mu(R.Y\cap \ga.B)\ \ge\ \frac1{t\eta}\, \mu(Y)\ \ge\ 2\,\mu(Y).
\end{eqnarray*}
If $\mu(A\setminus R.Y)\le \eta\, \mu(A)$ then, by $T.A\supset B$ and $\mu(A)=\mu(B)\in (0,\infty)$, we have 
\begin{eqnarray*}
	\mu(B\setminus N(Y))&\le& \mu(B\setminus TR.Y)\ \le\ \mu(T.(A\setminus R.Y))\\
	&\le& t\,\mu(A\setminus R.Y)\ \le\ t\,\eta\,\mu(A)\ <\ \frac12\,\mu(B).
	\end{eqnarray*} 
Thus, $G$ is indeed a bipartite $1$-expander.

Now, by Theorem~\ref{th:LN}, $G$ contains a Borel matching $M$ 
such that the unmatched sets $A\setminus \pi_1(M)$ and $B\setminus \pi_2(M)$ have measure 0. By Lemma~\ref{lm:MEQ}\ref{it:MEQ2}, the matched sets $\pi_1(M)\subset A$ and $\pi_2(M)\subset B$ are Borel equidecomposable, as desired.
\end{proof}

Let us state one step needed in the proof of Theorem~\ref{th:suff}\ref{it:suff2} as Part~\ref{it:combine1}  of the following auxiliary proposition.

\begin{prop}\label{pr:combine}
	Let $\Ga\actson \Om$ be as in Assumptions \ref{as:1} and~\ref{as:mu}. Let measurable sets $A, B\subset \Om$ be essentially Borel equidecomposable.	
	\begin{enumerate}[(i),nosep]
		\item\label{it:combine1} If $A$ and $B$ are set-theoretically equidecomposable then $A$ and $B$ are measurably equidecomposable.
		\item\label{it:combine2} If $A$ and $B$ are Baire equidecomposable then $A$ and $B$ are Baire-Lebesgue equidecomposable.
	\end{enumerate}
\end{prop}
\begin{proof}
	We will prove both items at the same time.
Pick null sets $N_A$ and $N_B$ such that $A\setminus N_A$ and $B\setminus N_B$ are Borel equidecomposable.	
	 Let $U_1,\ldots, U_m\in \cal B$ and $\ga_1,\ldots, \ga_{m}\in \Ga$ be such that $A\setminus N_A = U_1\sqcup\ldots \sqcup U_{m}$ and $B\setminus N_B = \ga_1.U_1\sqcup \ldots\sqcup \ga_m.U_m$. Similarly let $V_1,\ldots, V_n\in 2^\Om$ and $\de_1,\ldots, \de_n\in \Ga$ be such that $A = V_1\sqcup \ldots\sqcup V_n$ and $B= \de_1.V_1\sqcup \ldots \sqcup\de_n.V_n$. 
	
	Let $N'$ be a Borel null set which contains $N_A\cup N_B$. Let $\Lambda$ be the subgroup of $\Ga$ generated by $\{\ga_1,\ldots, \ga_m,\de_1,\ldots, \de_n\}$ and let $N := \Lambda.N'$. Since $\Lambda$ is countable and $N'$ is a Borel null set, we have that $N$ also is a Borel null set. Furthermore, by the $\La$-invariance of the set $N$, we have that  $\de_i.(V_i\cap N)=\de_i.V_i\cap \de_i.N= \de_i.V_i\cap N$ for $i\in[n]$  and  $\ga_i.(U_i\setminus N) = \ga_i.U_i\setminus N$ for $i\in [m]$.
	
	It follows that $A$ and $B$ are  equidecomposable using the partition 
	$$
	 A = (U_1\setminus N)\sqcup\ldots \sqcup (U_m\setminus N)\sqcup (V_1\cap N) \sqcup \ldots \sqcup (V_n\cap N)
	 $$ and the group elements $\ga_1, \ldots, \ga_m, \de_1,\ldots, \de_n$. Informally speaking, we use the pieces $V_i$ on $N$ and the pieces
	$U_i$ on the complement of~$N$.
	
It is clear that the sets $U_i\setminus N$ are Borel (in particular they are measurable and have the property of Baire). Furthermore, the sets $V_i\cap N$ are contained in the null set $N$, so they are measurable. This finishes the proof of the first part. 
	
	Additionally, if the sets $V_i$ have the property of Baire then the sets $V_i\cap N$ have the property of Baire as well since $N$ is Borel. This observation finishes the proof of the second part.
\end{proof}

\begin{proof}[Proof of Theorem~\ref{th:suff}\ref{it:suff2}.]
	Let us show the converse direction (as the forward direction is trivial).

Since $A$ and $B$ are equidecomposable, they also cover each other. By the converse direction in Theorem \ref{th:suff}\ref{it:suff1} that we have already proved, these sets are essentially Borel equidecomposable. Thus Proposition~\ref{pr:combine}\ref{it:combine1} gives 
the required measurable equidecomposition of $A$ and $B$.\end{proof}

\section{Proof of Proposition~\ref{pr:paradox}}\label{se:aux}

There was some freedom as to how to define when an action is paradoxical. Our Definition~\ref{de:paradox}\ref{de:paradox3} is of the same form as the known paradoxes are usually stated. 
We had to add some further properties, namely, the local compactness (to ensure that the family of sets covered by the
above statement is rich enough) and the preservation of compact closures (a rather weak restriction
which holds, for example, if the whole space $\Omega$ is compact or the group acts by homeomorphisms). Note
that the latter property implies that the family $\CU$ is invariant. (Recall that $\CU$ consists of those subsets of $\Om$ that cover a non-empty open set and have compact closure.)

The proof of Proposition~\ref{pr:paradox} will occupy this section.
We will need the Banach-Schr\"oder-Bernstein Lemma (see e.g.\ \cite[Theorem~3.6]{TomkowiczWagon:btp}), stated in terms of equidecompositions.

\begin{lemma}[Banach-Schr\"oder-Bernstein]\label{lm:BSB} 
	If $A$ is equidecomposable to some $B'\subset B$ and $B$ is equidecomposable to some $A'\subset A$, then $A$ and $B$ are equidecomposable.\qed
\end{lemma}

Also, it will be convenient to use the  \emph{semigroup of equidecomposability types} $\C S$ that was introduced by Tarski~\cite{Tarki38}. For more details, see~\cite[Chapter~10]{TomkowiczWagon:btp} whose presentation we follow.
Informally speaking, we consider all multi-subsets of $\Om$ of bounded multiplicity, identified under  the appropriately defined equidecomposability relation. Formally, let $\C P^*(\Om)$
be the family of all subsets of $\Om^*:=\Om\times \I N$ whose projection on $\I N$ is finite. 
Let $\Ga^*$ be the direct product of $\Ga$
and the group 
$S_{\I N}$ of all permutations of $\I N$. Define its action $a^*$ on $\Om^*$ by $a^*((\gamma,\sigma),(x,n)):=(a(\gamma,x),\sigma(n))$ (that is, $\Ga^*$ acts component-wise). For $A,B\in\C P^*(\Om)$, we write $A\preceq B$ if $A$ is set-theoretically equidecomposable under the action $a^*$ to a subset of $B$. If $A\preceq B$ and $B\preceq A$, then we write~$A\sim B$. This is an equivalence relation.
The equivalence class of $A\in\C P^*(\Om)$ is denoted by $[A]$. Let
 \beq\label{eq:CS}
 \C S:=\{\,[A]\mid A\in\C P^*(\Om)\,\},
\eeq
 and define the sum of $[A],[B]\in\C S$ by taking disjoint representatives $A'\in [A]$ and $B'\in[B]$ and letting $[A]+[B]:=[A'\cup B']$.
 It is easy to see that this is well-defined and satisfies various natural properties, like commutativity, associativity, etc.  Also, $\preceq$ gives a partial order on~$\C S$.
We say that $A\in \C P^*(\Om)$ \emph{can be doubled} if $[A]\sim 2[A]$, where $n[A]:=[A]+\dots+[A]$ denotes the sum of $n$ copies of $[A]\in\C S$.
Under the identification of $x\in \Om$ with $x^*:=(x,0)\in\Om^*$, these definitions also apply to subsets of~$\Om$.

\begin{lemma}\label{lm:double}  If $A,B\subseteq \Om$ cover each other and $B$ can be 
	doubled, then $A$ can be 
doubled.
\end{lemma}

\begin{proof} 
By the covering assumption, $[A]\preceq m[B]$ and $[B]\preceq n[A]$
for some $m,n\in\I N$. From 
$[B]=2[B]$, it follows that
$[B]=r[B]$ for every integer $r\ge 2$. We conclude that
 \begin{equation}\label{eq:2nA<nA}
n(2[A])=2n[A]\preceq 2n(m[B])=2mn[B]=[B]\preceq n[A].
\end{equation}

A version of the \emph{Cancellation Law} (see~\cite[Theorem~10.20]{TomkowiczWagon:btp}) states that for any $Y,Z\in\C P^*(\Om)$ if $n[Y]\preceq n[Z]$ then $[Y]\preceq [Z]$. In brief, its proof proceeds by considering the bipartite graph $G$ whose parts are the sets $Y,Z\subset\Om\times \I N$ and whose edge set corresponds to a fixed witness of $n[Y]\preceq n[Z]$. Thus every vertex in $Y$ (resp.\ $Z$) has $G$-degree exactly $n$ (resp.\ at most~$n$). A simple double-counting of edges shows that every finite subset $X$ of $Y$ has at least $|X|$ neighbours in $Z$. Now, Rado's theorem~\cite{Rado42} (that uses the Axiom of Choice) gives that $G$ has a matching covering all elements of $Y$ and thus $Y\preceq Z$, as claimed.
	
 Thus, by~\eqref{eq:2nA<nA},  we have that $2[A]\preceq [A]$ which, by the Banach-Schr\"oder-Bernstein Lemma (Lemma~\ref{lm:BSB}), gives that $A$ can be doubled.\end{proof}

\begin{proof}[Proof of Proposition~\ref{pr:paradox}\ref{it:paradox1}.] Assume that $|\Om|>1$ as otherwise there is nothing to do. Recall that $A\in\CU$, that is, $A$  has compact closure
	and covers a non-empty open set. 
	
	One direction (namely that if $B$ is equidecomposable to $A$, then $B\in\CU$) was proved
	in Lemma~\ref{lm:minimal}.

Next, let us show that every $A\in\CU$ can be doubled. First, suppose that $A$ is an open set.
The set $A$ has to contain at least two elements. (Otherwise,  pick distinct elements $x,y\in\Om$ and an open set $U\ni x,y$ with compact closure,
and note that the sets $A$ and $U$ contradict the paradoxicality
of the action.) Since $|A|\ge 2$, we can find  disjoint non-empty open subsets $U,W\subset A$. 
	Since the action is paradoxical, $A$ is equidecomposable to each of $U$ and $W$. Thus two copies of $A$ are equidecomposable to $U\sqcup W\subset A$. By Lemma~\ref{lm:BSB}, $A$ can be doubled, as desired.
	For general $A\in\CU$, pick any open $U\in\CU$, which exists by the local compactness of~$\Om$. As we have just argued, the open set $U$ can be doubled. By Lemma~\ref{lm:minimal},
	$A$ and $U$ cover each other. Thus, by Lemma~\ref{lm:double}, $A$ can be doubled.

	Finally, it remains to show that arbitrary $A,B\in\CU$ are equidecomposable. 
	By Lemma~\ref{lm:minimal} we know that $A$ and $B$ cover each other. As we proved in the previous paragraph, each element $U\in \CU$ can 
	be doubled. By Lemma~\ref{lm:minimal}, each of the two obtained copies of $U$  necessarily belongs to $\CU$. Thus by induction, for any integer $n\ge 2$, there is a partition
	$A=A_1\sqcup\ldots\sqcup A_n$ such that each $A_i$ is equidecomposable to $A$ (and belongs to $\CU$). By choosing $n$ such that $n$ copies of
	$A$ cover $B$, one can conclude that $B$ is equidecomposable to a subset of $A$. The same also holds when we swap the
	roles of $A$ and $B$. Now the desired conclusion follows from Lemma~\ref{lm:BSB}.
\end{proof}

In order to establish the second part of Proposition~\ref{pr:paradox}, we rely on the following weaker version of a powerful result of Marks and Unger~\cite[Theorem~1.3]{MarksUnger16}.

\begin{theorem}[Marks and Unger~\cite{MarksUnger16}]
	\label{th:MU} 
 Suppose that $G=(V_1,V_2,E)$ is a locally finite bipartite Borel graph on a Polish space 
 such that
 \begin{equation}\label{eq:2Hall}
 |N_G(X)| \ge 2|X|,\quad \mbox{for every finite set $X$ with $X\subset V_1$ or
$X\subset V_2$.}
\end{equation} 
Then there is a Borel matching $M$ in $G$ such that the set of unmatched vertices in each part is meager.\qed
\end{theorem}

\begin{proof}[Proof of Proposition~\ref{pr:paradox}\ref{it:Paradox2}.] Again, the forward implication is trivial so we prove the converse direction only.
Let $A,B\in\CU\cap \C T$ be arbitrary. By Part~\ref{it:paradox1} of the proposition, that we have already proved, these sets are equidecomposable. Let $S\subseteq \Ga$ be a finite set that suffices for this equidecomposition. 
By enlarging $S$, assume that $S^{-1}=S\ni e$. 
Also, as we have showed in the proof of Part~\ref{it:paradox1}, every set in $\CU$ can be set-theoretically doubled. By repeating, we can find a finite symmetric set $T\ni e$ such that the relations $2|S|\,[A]\preceq [A]$ and 
$2|S|\,[B]\preceq [B]$ can be shown by using elements from~$T$ only. Let $R:=ST\cup (ST)^{-1}$ and consider the bipartite graph $G:=(A,B,E_{R}\cap (A\times B))$.

This graph satisfies~\eqref{eq:2Hall}. Indeed, take some finite set $X$, say $X\subset A$.
By the choice of $T$, the set $Y:=T.X\cap A$ contains $2\,|S|$ disjoint copies of $X$ and thus $|Y|\ge 2\,|S|\,|X|$. Since $Y\subset A\subset S.B$, there is $\ga\in S$ with $|Y\cap \ga.B|\ge |Y|/|S|\ge 2\,|X|$. Thus $N(X)\supset \ga^{-1}.(Y\cap \ga.B)$ has at least $2\,|X|$ elements, as claimed.

We cannot apply Theorem~\ref{th:MU} yet, as the sets $A$ and $B$ need not be Borel. So we proceed as follows.
Find Borel sets $A'\subseteq A$ and $B'\subseteq B$ such that $A\setminus A'$ and $B\setminus B'$ are meager and choose a Borel meager set $N$ containing $(A\setminus A')\cup (B\setminus B')$. (In fact, we can require $N$ to be an $F_\sigma$-set, see e.g.~\cite[Proposition~8.23]{Kechris:cdst}.) Let $\Lambda\subset\Ga$ be the subgroup generated by $R$. Of course, $\Lambda$ is countable. Since the action preserves meager and Borel sets, the set $\Lambda.N$ is meager and Borel. When we remove the set $\Lambda.N$ from the graph $G$, we remove whole components. Thus the new graph $G'$ still satisfies~\eqref{eq:2Hall}. Also, $G'$ is a Borel graph since its parts, $A\setminus\Lambda.N= A'\setminus \Lambda.N$ and $B\setminus\Lambda.N=B'\setminus\Lambda.N$, are Borel. Thus Theorem~\ref{th:MU} applies to $G'$ and gives a Borel matching $M$ such that all unmatched vertices are inside some Borel meager set~$N'$. By enlarging the set $N'$, we can also assume that it contains~$N$ as a subset. The Borel meager set $\Lambda.N'$ is again a union of some components of $G$. Thus the set-theoretic equidecomposition between $A$ and $B$ gives an equidecomposition between the meager sets $A\cap \Lambda.N'$ and $B\cap \Lambda.N'$ while $M$ gives
a Borel equidecomposition between $A\setminus \Lambda.N'$ and $B\setminus  \Lambda.N'$. 
Putting these together, we get the required Baire equidecomposition between $A$ and $B$ (using the finite set $R\subseteq \Ga$).\end{proof}

\begin{proof}[Proof of Corollary~\ref{cr:Leb}.]
All forward implications in Corollary~\ref{cr:Leb} are trivial so let us show the converse direction. In all cases, the sets $A$ and $B$ are in $\CU$ (resp.\ essentially are in $\CU$) so by Lemma~\ref{lm:minimal} they  cover (resp.\ essentially cover) each other. 

Now, Item~\ref{it:EssBor} of the corollary is a direct consequence of
Theorem~\ref{th:suff}\ref{it:suff1}. In order to derive Item~\ref{it:Mes} from Item~\ref{it:EssBor}, it suffices by Proposition~\ref{pr:combine}\ref{it:combine1} to show that $A$ and $B$ are set-theoretically equidecomposable, and this follows from Proposition~\ref{pr:paradox}\ref{it:paradox1}. 

Finally, let us derive Item~\ref{it:BL} from Item~\ref{it:Mes}. By the latter, we know that $A$ and $B$ are measurably equidecomposable. Proposition~\ref{pr:paradox}\ref{it:Paradox2} gives that these sets are also Baire equidecomposable.
These two equidecompositions can be combined into a Baire-Lebesgue equidecomposition by Proposition~\ref{pr:combine}\ref{it:combine2}.\end{proof}

\section{Relation to (local) spectral gap}\label{spectral}\label{sec-spectral-gap-and-expansion}

Recall that the local spectral gap property was defined in Section~\ref{intro:lsg}.
Note that we do not assume in the definition that $\Ga$ is countable. The authors of \cite{BoutonnetIoanaSalehi17} consider only countable $\Ga$, in which case our definition coincides with the one in~\cite{BoutonnetIoanaSalehi17}. However, the definition of local spectral gap, as stated in Section~\ref{intro:lsg}, makes sense also when $\Ga$ is not countable.

For a finite multiset $Q\subset \Ga$, let $T_Q\colon L^2(\Om,\mu) \to L^2(\Om,\mu)$ be the averaging operator defined by 
 \begin{equation}\label{eq:TQ}
(T_Qf)(x):=\frac1{|Q|} \sum_{\gamma\in Q} f(\gamma^{-1}. x),\qquad f\in L^2(\Omega,\mu),\ x\in\Omega.
 \end{equation}
We say that the operator $T_Q$ has \emph{spectral gap} if there exists a 
constant $c>0$ such that for any $f\in L^2(\Om,\mu)$ with $\int_\Om f(x) 
\dd\mu(x) =0$ we have $\|T_Q f\|_2 \le 
(1-c)\,\|f\|_2$. Also, we say that the action $a:\Ga\actson \Om$ \emph{has spectral gap} if there is a finite multiset (equivalently, a finite set) $Q$ such that the operator $T_Q$ has spectral gap. If $\mu(\Omega)<\infty$, then the latter property is easily seen to be equivalent to the local spectral gap of the action $a$ with respect to the whole space $\Omega$. 

Boutonnet et al~\cite[Theorem A]{BoutonnetIoanaSalehi17} proved the following sufficient condition for local spectral gap.

\begin{theorem}[Boutonnet et al~\cite{BoutonnetIoanaSalehi17}]\label{th:BIS:A} Let $\Gamma$ be a connected  Lie group with
	a fixed left Haar measure $m_\Gamma$.  Suppose that 
	the Lie algebra $\frak g$ of $\Gamma$ is simple. Let $\Ad:\Gamma\to GL(\frak g)$ denote its adjoint
	representation. 
	Let $\Lambda$ be a dense countable subgroup of $\Gamma$ and let $\frak B$ be a
	basis of $\frak g$ such that for every $g\in \Lambda$ the matrix of $\Ad(g)$ in the basis
	$\frak B$ has all entries algebraic. Then the left translation action $\Lambda\actson (\Gamma,m_\Gamma)$ has local spectral gap with respect to every
	measurable set $X$ with compact closure and non-empty interior.\qed
\end{theorem}

The above notions and results are of interest to us because of the following equivalence, mentioned in
the Introduction.

\begin{lemma} \label{lm:equivalence} Let Assumptions \ref{as:1} and~\ref{as:mu} apply and
	let $X\subset \Om$ be a measurable set of finite positive measure. Then the following are equivalent.
	\begin{enumerate}[(A),nosep]
		\item\label{it:equivalenceA} The set $X$ is a domain of expansion with respect to the action $\Ga\actson \Om$.
		\item\label{it:equivalenceB} The action $\Ga\actson \Om$ has  local spectral gap with respect to $X$.
	\end{enumerate}
\end{lemma}

\begin{proof} Boutonnet et al~\cite[Theorem 7.6]{BoutonnetIoanaSalehi17} proved that, under the additional assumptions that $\Ga$ is countable and the action $\Ga\actson\Omega$ is ergodic,
 Property~\ref{it:equivalenceB} (i.e., the local spectral gap with respect to $X$) is equivalent to the following:

 {\it	\begin{enumerate}[(C),nosep]
\item \label{it:equivalenceC}  If a sequence $A_n$, $n=0,1,\ldots$, of measurable subsets of $\Om$ satisfies  $\mu(A_n\cap X)>0$ for all $n$, and 
\beq\label{eq-lim}
\lim_{n\to\infty}\, \frac{\mu\big((\ga.A_n\bigtriangleup A_n\big)\cap X)}{\mu(A_n\cap X)} =0
\eeq
 for all $\ga\in \Ga$, then $\lim_{n\to\infty} \mu(A_n\cap X) = \mu(X)$.\end{enumerate}
}

In fact, the part of the proof in~\cite{BoutonnetIoanaSalehi17} that shows the equivalence between \ref{it:equivalenceB} and \ref{it:equivalenceC} does not use the ergodicity of the action $\Gamma\actson\Omega$. Clearly, \ref{it:equivalenceA} looks closer in spirit to \ref{it:equivalenceC} than to \ref{it:equivalenceB} and, indeed, it is fairly easy to derive the equivalence of \ref{it:equivalenceA} and \ref{it:equivalenceC}. Since we need only the implication \ref{it:equivalenceB} $\implies$ \ref{it:equivalenceA} for our equidecomposition results, we  present a direct proof of this for reader's convenience. In the other direction,
we show only that \ref{it:equivalenceA} implies \ref{it:equivalenceC}, leaving to the reader to check that the proof of the implication \ref{it:equivalenceC} $\implies$ \ref{it:equivalenceB} from~\cite[Theorem 7.6]{BoutonnetIoanaSalehi17} does not use ergodicity.


In order to prove \ref{it:equivalenceB} $\implies$ \ref{it:equivalenceA}, we need some preliminaries. 
Let us say that the action
$\Ga\actson \Omega$ has a \emph{highly local spectral gap} with respect to
a measurable set $B\subset \Omega$ with $0<\mu(B)<\infty$ if there are
a finite set $S\subset \Ga$ and a real $\kappa$ such that
\begin{equation}\label{eq:rlsg}
\|f\|_{2,B}\le \kappa \sum_{g\in S} \|g.f-f\|_{2,B\,\cap\, g.B}
\end{equation}
for any $f\in L^2(\Omega,\mu)$ with $\int_B f\dd\mu=0$ (equivalently, for every $f\in L^2(B,\mu)$ with $\int_B f\dd\mu=0$).

One motivation behind this definition (besides that it is useful for our proof) is that each side of~\eqref{eq:rlsg} depends only on the restriction of $f$ to $B$ (but not on any values of $f$ outside $B$ as is the case in~\eqref{eq-local-gap}).

Clearly, the highly local spectral gap for $B$ implies the local spectral gap for 
$B$ (with the same choice of $S$ and $\kappa$). Let us show that the converse implication 
also holds. 

\begin{claim}\label{cl:RSG} If the action has local spectral gap with respect to a measurable set $B\subseteq \Om$ with $0<\mu(B)<\infty$, then it has the highly local spectral gap property with respect to~$B$.\end{claim}

\bpf[Proof of Claim.] Let a finite set $S$ and a real $\kappa$ satisfy the local 
spectral gap property with respect to~$B\subset \Om$.  By enlarging $S$, we can assume that $e\in S$.
Define $\kappa':=\kappa\, |S|$ and
$$
S':=\{gh^{-1}\mid g,h\in S\},
$$
where we view $S'$ as a set, not as a multiset. (Alternatively, if one views $S'$ as a multiset so e.g.\ $|S'|=|S|^2$, then $\kappa'=\kappa$ suffices in the argument below.)

Let us show that $S'$ and $\kappa'$ satisfy
the highly local spectral gap condition. Take any $f\in 
L^2(\Omega,\mu)$ 
with 
$\int_B f\dd\mu=0$.

We define $f'\in L^2(\Omega,\mu)$ as follows. Fix some total order on $S$ with the identity $e$ being the smallest element. For $x\in \Omega$,  if there is $\gamma\in S$ such that $\gamma.x\in 
B$ then let $\gamma_x$ be the smallest such $\gamma$ and define $f'(x):=f(\gamma_x.x)$; otherwise (i.e.,\ if $S.\{x\}\cap B=\emptyset$), we let $f'(x):=0$ (while $\gamma_x$ is undefined). Since $e$ comes before any other element of $S$, we have that $f'(x)=f(x)$
for all $x\in B$. For $g,h\in S$, define 
$$
 B_{g,h}:=\{x\in B\mid \gamma_{g^{-1}.x}=h\}.
 $$ 
 It follows from the definition that, for every $g\in S$, we have $B=\sqcup_{h\in S} B_{g,h}$, that is,  
 the sets $B_{g,h}$, $h\in S$, are disjoint and partition $B$. Also, trivially, $B_{g,h}\subseteq B\cap gh^{-1}.B$.

As $f$ and $f'$
coincide on $B$, we have
$\int_B f'\dd\mu=\int_B f\dd\mu=0$. Also, the square of the $L^2(\Omega,\mu)$-norm of 
$f'$ is 
finite, as it is at most $|S|$ times the square of $\|f\|_{2,B}$. By the properties stated above and the inequality
$(\sum_{h\in S} x_h)^{1/2}\le \sum_{h\in S} x_h^{1/2}$ valid for any non-negative reals $x_h$, we have that
\begin{eqnarray*} \|f\|_{2,B} &=&
	\|f'\|_{2,B}\ \le\ \kappa \sum_{g\in S} \|g.f'-f'\|_{2,B}\\
	&=&  \kappa \sum_{g\in S} 
	\left(\sum_{h\in S}\int_{B_{g,h}}  
	(f(hg^{-1}.x)-f(x))^2\dd\mu(x)\right)^{1/2}\\
	&\le &  \kappa \sum_{g\in S} 
	\left(\sum_{h\in S}\int_{B\,\cap\, gh^{-1}.B} 
	(f(hg^{-1}.x)-f(x))^2\dd\mu(x)\right)^{1/2}\\
	&\le & \kappa \sum_{g\in S} \sum_{h\in S}
	\left(\int_{B\,\cap\, gh^{-1}.B} 
	(f(hg^{-1}.x)-f(x))^2\dd\mu(x)\right)^{1/2}\\
	&\le& \kappa\,|S| \sum_{\gamma\in S'}\left(\int_{B\,\cap\,
		\gamma.B}  
	(f(\gamma^{-1}.x)-f(x))^2\dd\mu(x)\right)^{1/2}\\
	&=& \kappa' \sum_{\gamma\in S'} \|\gamma.f-f\|_{2,B\,\cap\, 
		\gamma.B},
\end{eqnarray*}
that is, the real $\kappa'$ and the set $S'$ establish the highly local spectral gap for~$B$.
This proves Claim~\ref{cl:RSG}.\epf

Now, we can give a direct proof that \ref{it:equivalenceB} implies~\ref{it:equivalenceA}. Since the properties in question are invariant under scaling
the measure by a constant factor, assume that $\mu(X)=1$. By Claim~\ref{cl:RSG}, we can find a real $\kappa$ and a symmetric finite set $S\subseteq \Gamma$ with $S\ni e$ that satisfy  the highly local spectral 
gap condition for $X$. Now, given $\eta>0$, let $\ell\in\I N$ satisfy $(1+\eta/(2\kappa^2|S|^2))^\ell\ge 1/\eta$ and consider the set
 $$
 Q:=S^\ell=\{\gamma_1\ldots\gamma_\ell\mid \gamma_1,\ldots,\gamma_\ell\in S\}
 $$ of all possible $\ell$-wise products of
elements of~$S$. We will show that this set $Q$ is $\eta$-expanding for~$X$. We need an auxiliary claim first.

\begin{claim}\label{cl:SY} For every measurable $Y\subset X$, we have
	\begin{equation}
	\label{eq:SY}
	\mu(S.Y\cap (X\setminus Y))\ge \frac{1-\mu(Y)}{2\kappa^2|S|^2}\,\mu(Y).
	\end{equation}
	\end{claim}

\begin{proof}[Proof of Claim.]
Let $y:=\mu(Y)$.  Define $f:\Om\to\I R$ by
\beq\label{eq:f}
f(x):=(1-y)\I 1_Y-y\I 1_{X\setminus Y}=\left\{\begin{array}{ll}
	1-y,&x\in Y,\\ -y, & x\in X\setminus Y,\\
	0,& x\in \Om\setminus X.
\end{array}\right.
\eeq
Then $\|f\|_{2,Y}^2=(1-y)^2y+y^2(1-y)=y(1-y)$ and $\int_X f(x)\dd\mu(x) =0$. 

Let $\gamma\in S$. For $x\in X\cap  \gamma.X$, we have that $\gamma^{-1}.x\in X$ and thus $|(\gamma.f)(x)-f(x)|=|f(\gamma^{-1}.x)-f(x)|$ assumes value 0  or $1$; in fact, it is 1 if and only if $x\in Y$ and $\gamma^{-1}.x\in X\setminus Y$ or vice versa, that is, precisely if 
$$
x\in (Y\cap \gamma.(X\setminus Y))\cup (\gamma.Y\cap(X\setminus Y)).
$$
Since $S=S^{-1}$, the latter set is a subset of $\gamma.Z\cup Z$, where $Z:=S.Y\cap (X\setminus Y)$.
Since the action is measure-preserving, we conclude that $\|\gamma.f-f\|_{2,X\,\cap\,  \gamma.X}^2\le \mu(\gamma.Z\cup Z)\le 2\mu(Z)$. 

The above (in)equalities and the choice of $S,\kappa$ give that
$$
(y(1-y))^{1/2}=\|f\|_{2,X}\le \kappa\sum_{\gamma\in S} \|\gamma.f-f\|_{2,X\,\cap\,  \gamma.X}\le \kappa|S|\, \big(2\mu(Z)\big)^{1/2},
$$
 which implies the claim.\end{proof}

Now we are ready to show that $Q$ is $\eta$-expanding for~$X$. Take an arbitrary measurable subset $Y\subset X$. Let $Y_0:=Y$ and, inductively for $i\in[\ell]$, let $Y_i:=S.Y_{i-1}\cap X$. Clearly, $Y_\ell\subset Q.Y\cap X$. If for some $i\le \ell$, we have
$\mu(Y_i)\ge 1-\eta$, then (since $e\in S$) we have that each of $Y_i\subset\ldots\subset Y_\ell\subset Q.Y\cap X$ has
measure at least $1-\eta$, as required. Otherwise, we obtain from Claim~\ref{cl:SY} by induction on $i=0,\ldots,\ell$ that $\mu(Y_i)\ge(1+\eta/(2\kappa^2|S|^2))^i\mu(Y)$. Taking $i=\ell$, we get the the required lower bound $\mu(Q.Y\cap X)\ge \mu(Y_\ell)\ge \mu(Y)/\eta$. Since $\eta>0$ was arbitrary, we conclude that $X$ is a domain of expansion. We have shown that \ref{it:equivalenceB} implies \ref{it:equivalenceA}.

Now, let us show that \ref{it:equivalenceA} implies \ref{it:equivalenceB}. Let $X$ be a domain of expansion. For each positive integer $n$, fix some finite set $S_n\subset \Ga$ which is $(1/n)$-expanding for~$X$. Let $\Lambda$
be the  subgroup of $\Ga$ generated by $\cup_{n=1}^\infty S_n$. 

First, let us show that the countable group $\Lambda$ satisfies~\ref{it:equivalenceC}. Suppose on the contrary that some sequence $(A_n)_{n\in\I N}$ violates this property. 
Let $A_n':=A_n\cap X$. By passing to a subsequence, we can assume that there is an integer $m\ge 2$ such that $\mu(A_n')< \mu(X)(1-2/m)$ for every~$n$.
Let $S:=S_m$, that is, $S$ is a $(1/m)$-expanding set for~$X$. Thus, for every $n\in\I N$,
$$
 \mu (S.A_n'\cap X) \ge \min\left(\left(1-\frac1m\right)\mu(X),\ m\,\mu(A_n')\right).
$$
 By passing to a subsequence again, we can assume that 
\begin{enumerate}[(i),nosep]
\item\label{it:C1} for all $n$ we have $\mu (S.A_n'\cap X) \ge \left(1-\frac1m\right)\mu(X)$, or 
\item\label{it:C2} for all $n$ we have $\mu (S.A_n'\cap X) \ge  m\, \mu(A_n')$. 
\end{enumerate}

If \ref{it:C1} holds then clearly $\mu((S.A_n'\cap X)\setminus A_n')\ge \left(1-\frac1m\right)\mu(X)-(1-\frac2m)\mu(X)= \mu(X)/m$ and thus, for some $\ga\in S$ and infinitely many $n$, we have $\mu((\ga.A_n'\cap X)\setminus A_n') \ge \mu(X)/(m|S|)$, and so also 
$$
\mu\big((\ga.A_n\setminus A_n)\cap X\big)
 =\mu((\ga.A_n\cap X)\setminus(A_n\cap X))
 \ge \frac{\mu(X)}{m|S|}
$$
for infinitely many $n$. This is in contradiction with the assumption \eqref{eq-lim} of Property~\ref{it:equivalenceC}.

Suppose now that \ref{it:C2} holds. Since $m\ge2$, we have $\mu((S.A_n'\cap X) \setminus A_n')\ge \mu(A_n')$ for all~$n$. Therefore for some $\ga\in S$ and infinitely many $n$ we have $\mu((\ga.A_n'\cap X) \setminus A_n') \ge \mu(A_n')/|S|$, and hence 
$$\limsup_{n\to\infty} \frac{\mu((\ga.A_n\bigtriangleup A_n)\cap X)}{\mu(A_n\cap X)} 
 \ge \limsup_{n\to\infty} \frac{\mu((\ga.A_n\cap X)\setminus( A_n\cap X))}{\mu(A_n')}
\ge \frac{1}{|S|},
$$
which again is a contradiction to~\eqref{eq-lim}.

Thus \ref{it:equivalenceC} holds for $X$ with respect to the action $\Lambda\actson\Omega$. By~\cite[Theorem 7.6]{BoutonnetIoanaSalehi17}, the action of the countable group $\Lambda$ on $\Omega$ has local spectral gap with respect to~$X$, that is, \ref{it:equivalenceB} holds for the group~$\Lambda$.  Of course, when we enlarge the group to $\Ga$, then
\ref{it:equivalenceB} still holds.
\end{proof}

The following proposition will be needed later, for estimating the number of pieces in some equidecompositions given by our proofs.

\begin{prop}\label{pr:Spectral} In addition to Assumptions~\ref{as:1} and~\ref{as:mu}, assume that $\mu$ is a finite measure. Let $S\subset \Ga$ be a finite symmetric multiset, and let $c\in (0,1)$ be such that for every $f\in L^2(\Om,\mu)$ with $\int f(x)\dd\mu(x)=0$ we have $\|T_Sf\|_2\le (1-c)\|f\|_2$. Define $c':= c(2-c)$. Then the following statements hold.

\begin{enumerate}[(i),nosep]
\item\label{it:Spectral1} For every measurable $Y\subseteq \Om$ it holds that
 $$
 \mu(S.Y)\ge \frac{\mu(Y)\mu(	\Om)}{(1-c)^2\mu(\Om)+2c\mu(Y)-c^2\mu(Y)}\ge \mu(Y)\left(1+c'\,\frac{\mu(\Om\setminus Y)}{\mu(\Om)}\right).
 $$

\item\label{it:Spectral2} Let $\eta>0$ and  $\ell\in \N$ be such that $(1+c'\eta)^\ell>1/\eta$. Then $S^\ell\subset\Ga$  
	is an $\eta$-expanding set for~$\Om$. 
\end{enumerate}
 \end{prop}

\begin{proof} By scaling the measure, we can assume that $\mu(\Om)=1$. Take any measurable $Y\subset \Om$ with  $y:=\mu(Y)>0$. Analogously to~\eqref{eq:f}, 
define $f:=(1-y)\I 1_Y-y\I 1_{\Om\setminus Y}$. We have $\|f\|_2^2=y(1-y)$ and $\int_\Om f(x)\dd\mu(x) =0$. 

Let $Z:=S.Y$ and $z:=\mu(Z)$. Clearly, for every $x\in \Om\setminus Z$ we have that $(T_Sf)(x)=-y$. By
the invariance of the measure,
it also holds that $\int_\Om T_Sf\dd\mu
=|S|^{-1}\sum_{\ga\in S} \int_\Om \ga.f\dd \mu
=0$. Under these constraints on $T_Sf$, its $L^2$-norm is
minimised when the function is constant on $Z$, that is, assumes the value $(1-z)y/z$ there.
Thus
 $$
 \|T_Sf\|_2^2\ge \left(\frac{(1-z)y}{z}\right)^2z+y^2(1-z)=\frac{y^2(1-z)}z.
 $$
 By the spectral gap property, the left-hand side is at most $(1-c)^2\|f\|^2_2=(1-c)^2y(1-y)$. Solving the obtained linear inequality in $z$, we obtain the first inequality of Part~\ref{it:Spectral1}. The second inequality is obtained by observing that by $c,y\in [0,1]$,
 $$\frac{1}{(1-c)^2+2cy-c^2y}-1=\frac{c(2-c)}{(1-c)^2+2cy-c^2y}\,(1-y)\ge c' (1-y).$$

We prove Part~\ref{it:Spectral2} similarly as we did after Claim~\ref{cl:SY}. Define $Y_0:=Y$. Inductively for $i\in\I N$, let $Y_{i+1}:=S.Y_i$. If for some $i\le \ell$, we have $\mu(Y_i)\ge 1-\eta$, then  $\mu(Y_\ell)\ge \mu(Y_{\ell-1})\ge \dots\ge \mu(Y_i)$ (since the set $S$ is non-empty and $\mu$ is invariant). Otherwise, the measure of each new set $Y_i$, $i\le \ell$, increases by factor at least $1+c'\eta$ by Part~\ref{it:Spectral1} and thus $\mu(S^\ell.Y)\ge \mu(Y)/\eta$ by the choice of~$\ell$.\end{proof}

\section{Proof of Theorem~\ref{th:DE}}\label{se:DE}

\subsection{Proof of Theorem~\ref{th:DE} for $\SO(n)\actson \I S^{n-1}$, $n\ge 3$}\label{se:SO}

Here, $\Omega$ is the sphere $\I S^{n-1}$ with the uniform probability measure~$\mu$. It was shown independently  by Margulis \cite{Margulis80} and Sullivan \cite{Sullivan81} for $n\ge 5$, and by Drinfel'd \cite{Drinfeld84} for $n=3,4$, that
the action $SO(n)\actson \sphere^{n-1}$ has spectral gap. By Proposition~\ref{pr:Spectral}, the whole space $\I S^{n-1}$ is a domain of expansion.

In order to finish the proof, it is enough to show that every $A\in\C B\cap\CU$ is a domain of expansion. By the definition of $\CU$, $A$ covers some non-empty open set and thus also covers the whole sphere~$\I S^{n-1}$. Of course, the sphere $\I S^{n-1}\supseteq A$ also covers $A$. Now, Lemma~\ref{lm:Cover} gives that $A$ is a domain of expansion.  (Alternatively, we could have just quoted Remark~\ref{re:DoE} for this step.) Thus the action $\SO(n)\actson \I S^{n-1}$ is indeed expanding.

\subsection{Proof of Theorem~\ref{th:DE} for $G_2\actson\I R^2$}

Recall that $G_2$ is the subgroup of affine bijections of $\I R^2$ generated
by the special linear group $\SL(2,\I Z)$ and all translations. In order to prove that the action $a:G_2\actson\I R^2$ is expanding, we consider another action~$b$ defined as follows.

Identify the torus $\I R^2/\I Z^2$ with $X:=[0,1)^2$ and let $\nu$ denote the uniform probability measure on Borel subsets of~$X$. 
The group $\SL(2,\I Z)$ acts naturally on $\I R^2$, commuting with the reduction of vectors modulo~$\I Z^2$. Thus we obtain the standard measure-preserving action $b:\SL(2,\I Z)\actson (X,\nu)$.
A classical result of Rosenblatt \cite{Rosenblatt81} states that $b$ has spectral gap. By Proposition~\ref{pr:Spectral}, the set $X$ is a domain of expansion for the action~$b$. 

\hide{Take an arbitrary bounded set $A\subset\I R^2$ that covers a non-empty open set. Since $G_2$ acts transitively on $\I R^2$ by homeomorphisms, $A$ also covers the compact set $\O X\subset \I R^2$. 
Since $G_2$ contains all translations, $X$ covers any bounded set, in particular, it covers~$A$.
Thus, by Lemma~\ref{lm:Cover}, }

By Remark~\ref{re:DoE},
it is enough to show that $X$ is a domain of expansion for
the paradoxical action~$a$.
Let $\pi\colon \R^2\to X$ be the natural projection that reduces each coordinate modulo~$1$. For $\gamma\in \textrm{SL}(2,\Z)$, let $\gamma'$ denote the corresponding element of $G_2$ under the natural inclusion of $\textrm{SL}(2,\Z)$ into~$G_2$. For example, the restriction of the composition $\pi\circ a(\gamma',\cdot)$ to $X$ is equal to $b(\gamma,\cdot)$, the $b$-action of~$\gamma$.

Take any $\eta>0$. Since $X$ is a domain of expansion for the action $b$, there exists an $\eta$-expanding finite set $S\subset\SL(2,\I Z)$ for $X$ under the action $b$. We construct
an $\eta$-expanding set $T$ for $X\subset\I R^2$ under the action~$a:G_2\actson\I R^2$ as follows.
For every $\gamma\in S$ and $(m,n)\in\I Z^2$ such that  $\gamma'.X$ intersects the square $[m,m+1)\times [n,n+1)$,  add the product $t_{m,n}^{-1}\gamma'\in G_2$ into $T$, where $t_{m,n}\in G_2$ is the translation by vector $(m,n)$.
Since $|S|<\infty$, the constructed set
$T$ is finite too. Furthermore, for every $\gamma\in S$ and $U\subseteq X$, we have that $\ga.U=\pi(\gamma'.U)\subseteq T.U\cap X$:
indeed, for every integer square intersecting $\gamma'.U$, the set $T$ contains the composition of $\gamma'$ with the integer translation moving this square back to~$X=[0,1)^2$. Thus,  $S.U\subset T.U\cap X$ and, if $U$ is measurable, then
 $$
 \mu(T.U\cap X)\ge \mu(S.U)
 \ge
 \min\big((1-\eta)\,\mu(X),\,{\mu(U)}/{\eta}\big),$$
 where the last inequality follows from the facts that~$S$ is $\eta$-expanding for $(X,\nu)$ under the action~$b$ and the measures $\nu$ and $\mu$ coincide on~$X\supset S.U$. Thus $T\subset G_2$ is an $\eta$-expanding set for~$X$. As $\eta>0$ was arbitrary, $X$  is a domain of expansion for the action $a:G_2\actson \I R^2$, as desired.

\subsection{Proof of Theorem~\ref{th:DE} for the hyperbolic space $\I H^n$}\label{se:Hn}

For an introduction to hyperbolic spaces see 
e.g.\ Bridson and Haefliger~\cite[Section~2]{BridsonHaefliger:msnpc} or Ratcliffe~\cite{Ratcliffe06fhm}. One representation
of $\I H^n$ (that we will use here) is to take the bilinear form 
\begin{equation}\label{eq:form}
\langle u, v\rangle_{n,1}:=-u_{n+1}v_{n+1}+\sum_{i=1}^n u_iv_i,\quad u,v\in\I R^{n+1},
\end{equation} 
identify $\I H^n$ with upper sheet of the hyperboloid  
$$
\C H:=\{u\in \I R^{n+1}\mid \langle u, u\rangle_{n,1}=-1\}
$$ (namely, the sheet where $u_{n+1}>0$), and
define the metric $d$ by $\cosh d(u,v)=-\langle u, v\rangle_{n,1}$ for $u,v\in \I H^n$. 
The group of isometries of $\I H^n$
can be identified with $O(n,1)_0$, the group of $(n+1)\times (n+1)$-matrices which leave
the bilinear form in~\eqref{eq:form} invariant 
and do not swap the two sheets of $\C H$; see~\cite[Theorem~2.24]{BridsonHaefliger:msnpc} or~\cite[Theorem~3.2.3]{Ratcliffe06fhm}.
The group $\iso(\I H^n)$ of orientation-preserving isometries of $\I H^n$ corresponds to the 
subgroup $\SO(n,1)_0$ of index 2 in $O(n,1)_0$,  which 
consists of 
matrices with determinant 1. The space $\I H^n$ is equipped with an isometry-invariant measure $\mu$ whose push-forward
under the projection on the first $n$ coordinates of $\I R^{n+1}$ has density 
\beq\label{eq:HnDensity}
\rho(x_1,\ldots,x_n):=(1+(x_1^2+\ldots+x_n^2))^{-1/2}
\eeq
with respect to the Lebesgue measure on~$\I R^n$.

\begin{proof}[Proof of Theorem~\ref{th:DE} for $\iso(\I H^n)\actson\I H^n$, $n\ge 2$.] First, we show that Theorem~\ref{th:BIS:A} applies to
	the group $\Ga=\iso(\I H^n)=\SO(n,1)_0$, obtaining a countable subgroup $\Lambda$. Then we show that the expansion of $\Lambda\actson (\Ga,m_\Ga)$, where $m_\Ga$ is a left Haar measure on $\Ga$,
	can be transferred to the action~$\Lambda\actson (\I H^n,\mu)$.
	
	The Lie algebra of $\SO(n,1)_0$ (or $\SO(n,1)$) is $\G{so}(n,1)$ which consists of $(n+1)\times (n+1)$-matrices 
	$M$ such that
	\begin{equation}\label{eq:son1}
	M^TI_{n,1}+I_{n,1}M=0,
	\end{equation} 
	where $I_{n,1}$ is the diagonal matrix having $n$ entries equal to 1 and the last entry equal to~$-1$.
	It is well-know that $\G{so}(n,1)$ is simple; in fact, a complete characterisation of simple real Lie algebras is
	available,
	see e.g.\ Knapp~\cite[Theorem~6.105]{Knapp:lgbi}.
	
	The linear system of equations~\eqref{eq:son1} has integer coefficients, so we can choose a basis $\G B$
	for $\G{so}(n,1)$ that consists of matrices with all entries rational.
	
	The adjoint $\Ad(\gamma)$ for $\gamma\in \SO(n,1)_0$ maps $M\in \G{so}(n,1)$ to the matrix product~$\gamma M\gamma^{-1}$. If a matrix $\ga$ has algebraic entries, then so does its inverse; it follows that the matrix of the map $\Ad(\ga)$ when expressed in the basis $\G B$ has all entries algebraic.
	Thus the only non-trivial remaining assumption of Theorem~\ref{th:BIS:A} is the existence of a countable dense set $X\subset \SO(n,1)_0$
	such that each $\gamma\in X$ as a matrix has algebraic entries (as then 
	we can take $\Lambda$ to be the subgroup generated by~$X$).
	We can identify $\SO(n,1)$ with the variety in $V\subset \I R^{(n+1)\times (n+1)}$ defined by the system of polynomials with
	integer coefficients, stating that the determinant is $1$ and the bilinear form in~\eqref{eq:form} is preserved.
	Since $\SO(n,1)_0$ is a connectivity component of $\SO(n,1)$, the existence of $X$ follows from the following general lemma.
	
	\begin{lemma} Let $f_1,\ldots,f_m\in\I Z[x_1,\ldots,x_r]$ be polynomials 
		with integer coefficients. Let 
		$$
		V:=\{x\in\I R^r\mid \forall i\in[m]\ f_i(x)=0\}
		$$ 
 be the real variety defined by these polynomials. Then the set of vectors in $V$ with all entries algebraic is dense in 
		$V$ (in the standard topology on $\I R^r$ generated by open Euclidean balls).
	\end{lemma}

	\begin{proof} Although this lemma has surely been proved before, we could not find a suitable statement (of the real case) anywhere in print. So we present our proof.
		
		One of the consequences of the Tarski-Seidenberg Theorem~\cite{Tarski51,Seidenberg54} is that the ordered field of reals $\I R$ and the ordered field $\I A$ of real algebraic numbers satisfy the same set of first-order sentences, where the language includes the constants $0$ and $1$, the multiplication and addition functions, the equality relation, and the binary order relation. For an $r$-vector $x$, we can express the statement that $f_i(x)=0$ for each $i\in[m]$
				 in first-order logic. 
		Thus, for every Euclidean ball $B\subseteq \I R^n$ whose centre and radius are rational, 
		$B\cap V=\emptyset$ if and only if $B\cap V\cap \I A^n=\emptyset$.
		
		Now, suppose on the contrary to the claim that we have some $x\in V$ and rational $r>0$ such that the radius-$r$ ball $B\subset \I R^r$ around $x$ has no algebraic points from~$V$. Pick a rational vector $x'\in\I R^n$ within distance $r/3$ from $x$ and let $B'$ be the ball around $x'$ of radius $r/2$. Then $B'\subseteq B$, so $B'\cap V$ has no algebraic points. However, $B'\cap V$ is non-empty as it contains $x$, a contradiction. 
	\end{proof}

	Thus, by Theorem~\ref{th:BIS:A} the action $b:\Lambda\actson(\Gamma,m_\Gamma)$ of the countable subgroup $\Lambda$ of $\Gamma$ generated by $X$ has local spectral gap with respect to every measurable subset of $\Gamma$ with compact closure and non-empty interior.
	We will show that the restriction $a'$ of the measure-preserving action $a:\Gamma\actson (\I H^n,\mu)$ to $\Lambda$ is expanding, thus finishing the proof.

	\renewcommand{\i}{f}
	Take any Borel $B\subset \I H^n$ which has non-empty 
	interior and compact closure.  Let us show that $B$ is a domain of expansion. 
	
	One can derive from the formula for the hyperbolic distance that
	the topology on $\I H^n$ is the induced topology from~$\I R^{n+1}$. Thus $B$
	is a bounded subset of~$\I R^{n+1}$. 
	Since the density $\rho$ is uniformly bounded (namely, by 1), 
	we have that $\mu(B)<\infty$. Also, $\mu(B)>0$ as $\rho>0$
	everywhere. 	
	
	Let $x_0:=(0,\ldots,0,1)\in \I H^n$ and define $\i: \Gamma\to 
		\I H^n$ by $\i(\gamma):=a(\gamma, x_0)$, that is, under the assumed identification
		$\I H^n\subset \I R^{n+1}$, $\i(\gamma)$ the application of the matrix $\gamma$ to the vector~$x_0$.
		Define 
	\begin{equation}\label{eq:B'=InverseB}
	B':=\i^{-1}(B)=\{\gamma\in \Gamma\mid \gamma.x_0\in B\}.
	\end{equation}
	As it is easy so see, $a:\Gamma\times\I H^n\to\I H^n$ is continuous. Thus $\i$ is also continuous and $B'\subset \Gamma$ is a Borel set. 	
	Let $U\not=\emptyset$ be the interior of $B$. Fix $u\in U$ and $\gamma_0\in \Gamma$ 
	with $\gamma_0.x_0=u$. Again, by the continuity of $\i$ there is an open set in 
	$\Gamma$ around $\gamma_0$ that lies entirely inside $B'$, so $B'$ has a non-empty 
	interior.
	
	Let us argue that the closure $\O{B'}$ of $B'$ is compact. Take an arbitrary infinite sequence $\gamma_1,\gamma_2,\ldots\in B'$. Viewing $\I H^n$ as a subset of $\I R^{n+1}$, fix some $n$ elements 
	$x_1,\ldots,x_n\in \I H^n$ so that the $n+1$ vectors $x_0,x_1,\ldots,x_n\in\I R^{n+1}$ are linearly independent.
	Each isometry $\gamma_i\in\iso(\I H^n)$ can be represented by a linear $\I H^n$-preserving transformation $\I R^{n+1}\to \I R^{n+1}$ given by 
	some matrix $M_i\in \SO(n,1)_0$. The images of the special point $x_0$ by $M_i$ are in the compact set $\O B$, so by passing to a subsequence we can assume that they converge to some $z_0\in\O B$. Since each $M_i$ is an isometry of $\I H^n$, the images of $x_1,\ldots,x_n$ all lie in some large ball in $\I H^n$ around $z_0$ (and thus in some large ball in $\I R^{n+1}$). Again by passing to a subsequence, we can assume that, for every $j\in[n]$, $M_ix_j$ converges to some $z_j$ as $i\to\infty$. Let $X$, $Z$, and $X_i$ for $i\in \I N$ be the $(n+1)\times (n+1)$-matrices with columns respectively $x_0,\ldots,x_n$, $z_0,\ldots,z_n$, and $M_ix_0,\ldots,M_ix_n$. 
	By the finite-dimensionality, we have that $X_i\to Z$ as $i\to\infty$. Also, by the choice of the vectors $x_j$, the matrix $X$ is invertible. Thus $M_i=X_iX^{-1}$
	converges to $M:=ZX^{-1}$ as $i\to \infty$. This limiting matrix $M\in \SO(n,1)_0$ has to belong to $\O {B'}$, so this set is indeed compact.
	
	Thus, Theorem~\ref{th:BIS:A} applies to the set $B'$, giving that the action of $\Lambda$ on $(\Gamma,m_\Gamma)$ has local spectral gap with respect to this set.   This means by Lemma~\ref{lm:equivalence} that $B'$ is a domain of expansion. Let
	$\eta>0$ be arbitrary. Thus there are $\gamma_1,\ldots,\gamma_m\in \Lambda$ such that the $\eta$-expansion property, as defined in~\eqref{eq:DoE},
	is satisfied for every Borel $X\subset B'$. Let us argue that the same isometries are $\eta$-expanding for~$B$. Take any Borel $Y\subset B$.
	Let $Y':=\i^{-1}(Y)\in\C B$. Note that the $\i$-preimage of $(\cup_{i=1}^m \gamma_i.Y)\cap B$ is exactly 
	$(\cup_{i=1}^m \gamma_i .Y')\cap B'$; indeed we have
	$$f^{-1}(\gamma_i.y)=\{\gamma\in\Ga\mid \gamma_i^{-1}\gamma.x_0=y\}=\gamma_i\,\{\beta\in\Gamma\mid \beta.x_0=y\}=\gamma_i\,f^{-1}(y),\quad\mbox{for all $y\in Y$}.
	$$
	 It remains to show is that $\i$ is measure-preserving. This is
	exactly the statement of, for example,~\cite[Lemma~4 of Section~11.6]{Ratcliffe06fhm} (which follows with some work from the uniqueness of the Haar measure).
	Thus $B$ is a domain of expansion.
	
	By Remark~\ref{re:DoE}, this finishes the proof of the case $\iso(\I H^n)\actson \I H^n$ of Theorem~\ref{th:DE}.\end{proof}

\subsection{Proof of Theorem~\ref{th:DE} for $\SL(2,\I R)\actson \I R^2\setminus\{0\}$}\label{se:SL2R}

Here, we follow the same strategy as in Section~\ref{se:Hn}. One new caveat is that, unlike for the hyperbolic space, the stabiliser $H$ of a point is not a compact subgroup. So we cannot define $B'$ by the direct analogue of~\eqref{eq:B'=InverseB} as its closure would not be compact. We get around this by, essentially, taking  a compact subset of positive measure in each coset of $H$ in a  continuous way. Also, we could not find a version of \cite[Lemma~4 of Section~11.6]{Ratcliffe06fhm} that we could just cite here, so we prove it (as well as a few other claims) via direct explicit calculations.

\begin{proof}[Proof of Theorem~\ref{th:DE} for $\SL(2,\I R)\actson \I R^2\setminus\{0\}$.] Recall that $m_\Ga$ is a left Haar measure on $\Ga=\SL(2,\I R)$. 
We view each element $\gamma\in \Ga$
as a $2\times 2$-matrix $(\gamma_{ij})_{i,j=1}^2$. It is easy to to find a countable dense subgroup $\Lambda$ of $\SL(2,\I R)$: just take the
subgroup of matrices with all entries rational.
The Lie algebra $\G{sl}_2(\I R)$ of $\SL(2,\I R)$ consists of $2\times 2$ matrices with trace zero and is well known to be simple. For the
basis $\G B$, one can take, for example,
$$
X:=\Matrix{cc}{0&1\\ 0& 0},\quad Y:=\Matrix{cc}{0&0\\ 1&0},\quad Z:=\Matrix{cc}{1&0\\ 0&-1}.
$$
For $\gamma\in \SL(2,\I R)$, the adjoint $\Ad(\gamma)$  maps a $2\times2$-matrix $M=(M_{ij})_{i,j=1}^2\in \G{sl}_2(\I R)$ to $\gamma M\gamma^{-1}$. Clearly, we have $M=M_{1,2}X + M_{2,1}Y+M_{1,1}Z$. If we write the linear map $\Ad(\gamma)$ in
the basis $\G B$, then each entry of the corresponding $3\times 3$-matrix is, in fact, a quadratic polynomial with integer coefficients in the entries of the matrix $\gamma$.
\hide{Indeed, we have to take the $(1,2)$ entry of $\ga Z \eta$, $\eta= \ga^{-1}$, which is
	$$
	(\gamma_{1,1},\ga_{1,2})\Matrix{cc}{1 & 0\\ 0 & -1} \Matrix{c}{\eta_{1,2}\\ \eta_{2,2}}=(\gamma_{1,1},-\ga_{1,2})\Matrix{c}{\eta_{1,2}\\ \eta_{2,2}}= \ga_{1,2}\eta_{1,2}-\ga_{1,2}\eta_{2,2}.
	$$
}%
Therefore, if $\gamma\in \Lambda$ then $\gamma$ (and thus the matrix of $\Ad(\ga)$) has all entries rational.

So the conclusion of Theorem~\ref{th:BIS:A} applies to the  action $b:\Lambda\actson (\Ga,m_\Ga)$. In order to derive that the action $a:\Ga\actson\Om$ is expanding, where $\Om=\I R^2\setminus\{0\}$, we need some preparation.
Define
\begin{equation}\label{eq:H}
H:=\{\ga\in \Ga\mid \ga e_1=e_1\}=\{M(u)\mid u\in\I R\},
\end{equation}
 where
  $$\mbox{$e_1:=\Matrix{c}{1\\0}$\ \ and\ \  $M(u):=\Matrix{cc}{1&u\\0&1}$.}
 $$
 
The map $M$ gives an isomorphism between the topological groups $(\I R,+)$
and~$H$. Let $\rho$ be
the Haar measure on $H$ with $\rho(I)=1$, where $I:=\{M(u)\mid 0\le u\le 1\}$, that is, $\rho$ is the push-forward of the Lebesgue measure on $\I R$ by~$M$.

Clearly, the map $\i:\Ga\to \Omega$, where $\i(\ga):=\ga e_1$ for $\ga\in \Ga$, is continuous and the pre-images under $\i$ of the points in $\Omega$ are exactly
the left cosets of $H$. Consider the map $\sigma:\Omega\to \Ga$ defined by
$$
\Matrix{c}{x\\y}\mapsto \Matrix{cc}{x & \frac{-y}{{x^2+y^2}}\\ y & \frac{x}{{x^2+y^2}}},\quad \mbox{for }\Matrix{c}{x\\y}\in \Om.
$$

The map $\sigma$ is continuous on $\Omega=\I R^2\setminus\{0\}$ and $\sigma$ is the right inverse of $\i$:  $\i\,\circ\, \sigma=\Id_{\Om}$. 
By e.g.~\cite[Theorem~5.1.5]{Ratcliffe06fhm}, the function $\phi:\Om'\to \Ga$, defined by $\phi(x,h)=\sigma(x)h$ is a homeomorphism, where we let $\Om':=\Om\times H$. 
The space $\Om'$ comes with the measure $\mu':=\mu\times \rho$, the product  of
the Lebesgue measure $\mu$ on $\Om\subset \I R^2$ and the Haar measure $\rho$ on $H$.

Let us show that the push-forward of $\mu'$ by $\phi$ 
is a constant multiple of the Haar measure $m_\Ga$ on~$\Ga$. By the uniqueness of
the Haar measure, it is enough to argue that, for every $\ga\in \Ga$, viewed as the map  $\Ga\to\Ga$ where $\beta\mapsto \ga.\beta$, the composition
$\ga':=\phi^{-1}\circ \ga\circ \phi$ from $\Om\times H$ to itself preserves the measure $\mu\times \rho$. In this concrete case, it is easy to show this by writing explicit formulas. Namely, the inverse of $\phi$ is
$$
\Matrix{cc}{x&c\\ y & d}\mapsto \left(\Matrix{c}{x\\y},M(u)\right),\quad \mbox{where } u:=\frac{c+y/({x^2+y^2})}x=\frac{d-x/({x^2+y^2})}y.
$$
It routinely follows that $\ga'((x,y),M(u))=(\ga.(x,y), M(u+F))$, where $F=F(\ga,x,y)$
does not depend on~$u$.
\hide{
	Indeed, with $a=x$ and $b=y$,
	\begin{eqnarray*}
		\ga(\phi((a,b),M(u)))&=&\Matrix{cc}{\ga_{11}& \ga_{12}\\ \ga_{21}&\ga_{22}} \Matrix{cc}{a & \frac{-b}{\sqrt{a^2+b^2}}\\ b & \frac{a}{\sqrt{a^2+b^2}}} \Matrix{cc}{1&u\\ 0& 1}\\
		&=&\Matrix{cc}{\ga_{11}& \ga_{12}\\ \ga_{21}&\ga_{22}}\Matrix{cc}{a & au-\frac{b}{\sqrt{a^2+b^2}}\\ b & bu+ \frac{a}{\sqrt{a^2+b^2}}}\\
		&=&\Matrix{cc}{\ga_{11}a+\ga_{12}b & (\ga_{11}a+\ga_{12}b)u-\frac{\ga_{11}b+\ga_{12}a}{\sqrt{a^2+b^2}}\\ \ga_{21}a+\ga_{22}b & ... } 
	\end{eqnarray*}
	so
	$$
	w=\frac{(\ga_{11}a+\ga_{12}b)u+(\ga_{21}a+\ga_{22}b)\sqrt{(\ga_{11}a+\ga_{12}b)^2+(\ga_{21}a+\ga_{22}b)^2}}{\ga_{11}a+\ga_{12}b} = u+ F.
	$$}
We see that the bijection $\ga'$ is the measure preserving map $a(\ga,\cdot)$ in the first coordinate. Also, if the first coordinate is fixed, then the second coordinate of $\ga'$ corresponds to some translation of $\I R$
under the parametrisation $M$ from~\eqref{eq:H}. It follows by an application of Tonelli's theorem that $\ga'$ is measure-preserving, as claimed. 

Thus, by scaling the measure $m_\Ga$, we can assume that $\phi$ is a measure-preserving homeomorphism between $(\Om',\mu')$ and $(\Ga,m_\Ga)$. 
For the rest of the proof, it will be more convenient to replace the  action $b$ with its conjugate by $\phi$. Namely,  let the measure-preserving action $b':\Lambda\actson (\Om',\mu')$ be defined by $b'(\ga,(x,h)):=\phi^{-1}(b(\ga,\phi(x,h)))$ for $(x,h)\in\Om\times H$ and $\ga\in \Lambda$. 
By the calculations in the previous paragraph, for every $x\in \Om$, there is  the (unique) map $\ga^x:H\to H$ so that $b'(\ga,(x,h))=(a(\ga,x),\ga^x(h))$;
moreover, each $\ga^x$ under the isomorphism $H\cong\I R$ corresponds to some translation of $\I R$ and thus preserves the Haar measure~$\rho$. 

Now, we are ready to show that the action $a:\Ga\actson \Om$ is expanding. 
Take any Borel $B\subset \Om$ with non-empty interior and
compact closure. Let $B':=B\times I$. 
Clearly, $B'\subset \Om'$ has compact closure and non-empty interior and $\mu'(B')=\mu(B)$.
Since $\phi$ is a measure-preserving homeomorphism, Theorem~\ref{th:BIS:A} (when applied to $b$) gives
that the conjugated action $b'$
has local spectral gap with respect to every measurable subset of $\Om'$ with compact closure and non-empty interior. 
Thus there are a real $\kappa$ and a finite set $S\subset\Lambda$ such that
\begin{equation}\label{eq:B'}
\|f'\|_{2,B'}\le \kappa \sum_{s \in S} \|s.f'-f'\|_{2,B'},\quad \mbox{for all $f'\in L^2(\Om',\mu')$ with $\int_{B'} f'\dd \mu'=0$.}
\end{equation}

Let us show that the same choice of $\kappa$ and $S$ also witnesses the local spectral gap of the action $a$ with respect to $B$.  Take any $f\in L^2(\Om,\mu)$ with
$\int_B f\dd \mu=0$.  The idea is to apply the inequality in~\eqref{eq:B'}
to the function $(x,h)\mapsto f(x)$ except, in order to have an $L^2$-function, we set it to 0 at the points which do not matter when we consider local spectral gap for $B'$ under $S$ (namely, those points that do not appear in~\eqref{eq:B'}). So, let $f'(x,h):=f(x)$ for $(x,h)\in B'\cup S^{-1}.B'$ and let $f'$ be zero otherwise. The obtained function $f'$ is in
$L^2(\Om',\mu')$ because is it obtained  by patching together  compositions of $g\in L^2(\Om',\mu')$ with finitely many measure-preserving maps $b'(\ga,\cdot)$, $\ga\in S$, where $g(x,h):=f(x)$ if $h\in I$ and is set to $0$ otherwise. Since 
$f'(x,h)=f(x)$ for all $(x,h)\in B'$, we have $\int_{B'} f'\dd \mu'=\int_{B} f\dd \mu=0$. Thus~\eqref{eq:B'} applies. 
Since 
$f'(x,h)=f(x)$ for $(x,h)\in S^{-1}.B'$ and  $\ga^x$ preserves the measure~$\rho$ for each $x\in\Om$, we have for each $s\in S$ that
\begin{eqnarray}
\|s.f'-f'\|_{2,B'}&=& \int_{B\times I} (f'(s^{-1}.(x,h))-f'(x,h))^2\dd\mu'(x,h)
\nonumber
\\
&=& \int_{B\times I} (f(s^{-1}.x)-f(x))^2\dd(\mu\times \rho)(x,h)\ =\ \|s.f-f\|_{2,B}.
\nonumber
\end{eqnarray}
It follows by~\eqref{eq:B'} that 
$$
\|f\|_{2,B}=\|f'\|_{2,B'}\le \kappa \sum_{s \in S} \|s.f'-f'\|_{2,B'}=\kappa \sum_{s \in S} \|s.f-f\|_{2,B}.
$$
Thus the action $a:\Ga\actson \Om$ is expanding by Lemma~\ref{lm:equivalence} (and Remark~\ref{re:DoE}), as required.
\end{proof}

\subsection{Proof of Theorem~\ref{th:DE} for $\iso(\I R^n)\actson \I R^n$, $n\ge 3$}\label{se:Rn}

We give a few proofs of this case: by deriving it from known results in Section~\ref{se:RnCite}
and then giving  a more direct proof in Section~\ref{se:direct}.

\subsubsection{Derivation from known results}\label{se:RnCite}

Theorem~\ref{th:DE}\ref{it:Rn} can be rather straightforwardly derived from known (deep) results in a few different ways.

\begin{proof}[First proof of Theorem~\ref{th:DE} for $\iso(\I R^n)\actson\I R^n$, $n\ge3$.] Margulis \cite{Margulis82} showed that, for this action, every isometry-invariant positive mean defined on compactly 
supported bounded measurable functions is a scalar multiple of the integral with respect to the
Lebesgue measure. Also, it can be shown via Lebesgue's density theorem that every dense subgroup of $\iso(\R^n)$ acts ergodically on~$\R^n$. Thus the implication (1) $\Longrightarrow$ (4) of \cite[Theorem~7.6]{BoutonnetIoanaSalehi17} applies here and gives that the action $\iso(\I R^n)\actson\I R^n$ has local 
spectral gap with respect to every measurable set in~$\CU$. We conclude by 
Lemma~\ref{lm:equivalence} that the action $\iso(\I R^n)\actson\I R^n$ is expanding.
\end{proof}

Alternatively, the desired local spectral gap property of $\mathrm{Iso}(\I R^n)\actson \I R^n$ 
can be derived from  \cite[Theorem~A]{BoutonnetIoanaSalehi17} (which is
Theorem~\ref{th:BIS:A} here) in a similar way as it was done for the action $\iso(\I H^n)\actson \I H^n$ in Section~\ref{se:Hn}. Yet another proof is to use the following more recent
result of Boutonnet and Ioana~\cite[Theorem~A]{BoutonnetIoana20}.

\begin{theorem}[Boutonnet and Ioana~\cite{BoutonnetIoana20}]
	\label{th:BoutonnetIoana20} 
	Let $\Lambda$ be a countable dense subgroup of $\iso(\I R^n)$, $n\ge 3$, such
	that the left-translation action $\Theta(\Lambda)\actson \SO(d)$ has spectral gap, where $\Theta:\iso(\I R^n)\to \SO(n)$ denotes the natural quotient. Then the natural action $\Lambda\actson \I R^n$ has local spectral gap with respect to every measurable set with compact closure and non-empty interior.\qed
\end{theorem}

In order to apply Theorem~\ref{th:BoutonnetIoana20}, one can let $\Lambda$ be the subgroup generated by a finite subset of $\SO(n)\subset\iso(\I R^n)$ having the spectral gap property and some countable dense subset of $\iso(\I R^n)$.

\subsubsection{Approach via more direct computations}\label{se:direct}

In this  subsection we will reprove the case of $\iso(\I R^n)\actson\I R^n$, $n\ge 3$, of Theorem~\ref{th:DE} in a more direct way. We prove the case $n=3$ first and derive the general case $n\ge 4$ as a consequence. We will not use the results of Margulis \cite{Margulis82} nor Boutonnet et al~\cite{BoutonnetIoanaSalehi17,BoutonnetIoana20}. However, when proving the base case $n=3$, we still need as an input the spectral gap property of the action  of $\textrm{SO}(3)$ on the $2$-dimensional sphere. While the original proof of this by Drinfel'd~\cite{Drinfeld84} requires a fair amount of background, there are more elementary proofs now: 
see, for example,\ Benoist and de Saxc\'e~\cite{BenoistDesaxce16}.


The rest of this section is devoted to domains of expansion in $\R^n$. However, we start by proving a sufficient criterion for being a domain of expansion for a general action $\Ga\actson \Om$. 

Let us informally motivate the upcoming definitions. Suppose that we fix some $\rho>1$ and would like to show that the \emph{annulus}  
$Y:=\{y\in\R^3\colon  1 \le \|y\|_2\le \rho\}$
is a domain of expansion. For $U\subset Y$ and $z\in[1,\rho]$ let $U_z:=\{y\in U\mid \|y\|_2=z\}$ be the \emph{$z$-leaf} of $U$. For $i\in [3]$, let $\mu^i$ denote the $i$-dimensional Hausdorff measure on $\I R^3$.

\newcommand{\D}{c}

By the spectral gap property of $\SO(3)\actson (\I S^2,\mu^2)$ and Proposition~\ref{pr:Spectral}, the sphere $\I S^2$ is a domain of expansion under this action. So, for every $\delta>0$, there is a finite $\delta$-expanding set~$S_\de\subset \SO(3)$. For each $z\in [1,\rho]$, the group $\SO(3)$ also acts on $(Y_z,\mu^2)$ (which is just the sphere of radius~$z$) and the set $S_\delta$ is also $\delta$-expanding for $Y_z$. Thus, for every Borel $U\subset Y$, the leaf $U_z$ under the action of $S_\delta$
occupies at least $(1-\delta)$-fraction of $Y_z$ or expands by factor at least~$1/\de$ in the measure~$\mu_z$. 
It follows that, if $U$ does not expand in measure under $S_\delta$, then this means
that the set $U$ is  ``close'' to a \emph{ring set} (a union of some spheres~$Y_z$). 

In order to deal with such sets, we add a finite set $T$ of isometries of $\I R^3$ with the property that there is $\D>0$ such that if $U$ is any Borel ring set then $\mu^2((T.U)_z)\ge \D\mu(U)$ for every~$z\in [1,\rho]$. We will call such a set $T$ a \emph{diffuser}. Informally speaking, a diffuser spreads any ring set across all radii of interest fairly uniformly. If we apply a spherical expanding set $S_\be$ with $\beta\ll \delta$ to any such ``uniformly spread'' set $T.U$, then we ensure that, for each $z\in [1,\rho]$, its $z$-leaf $(T.U)_z$ expands or occupies most of $Y_z$. It follows from the above properties that $S_\be.T$ has a good expansion property when applied to any (almost) ring set. Thus, our proof strategy is as follows: if $U$ expands under $S_\delta$ then we are done; otherwise $S_\delta.U$ is close to a ring set and 
consequently expands under $S_\be T$. 
The exact value of $\rho$ will be chosen to simplify finding a diffuser set; in fact, we will choose $\rho$ so that the diffuser can be taken to be a one-element set (see  Figure~\ref{fig-diffuser} and Lemma~\ref{lm:CubeDiffuser}).

Let us give all formal general definitions (that are motivated by the above discussion). Assumptions~\ref{as:1} and~\ref{as:mu} apply everywhere in this section.

\begin{definition}\label{de:foliation}
A \emph{foliation} of a Borel set $Y\subset \Omega$ is a pair $((Z,\C B_Z,\nu),(Y_z,\mu_z)_{z\in Z})$, where 
 \begin{enumerate}[(i),nosep]
 	\item\label{it:foliation1} $(Z,\C B_Z,\nu)$ is a standard measure space with $0<\nu(Z)<\infty$, 


\item\label{it:foliation2} the sets $Y_z$, $z\in Z$, are pairwise disjoint subsets of $Y$,

\item\label{it:foliation3} for every Borel $X\subset Z$, the set $Y_X:=\cup_{z\in X} Y_z$ is a Borel subset of $Y$,

\item\label{it:foliation4} each $\mu_z$ is a finite measure on $(\Omega,\C B)$ supported on $Y_z$ such that for every Borel set $U\subset Y$ the function $z\mapsto \mu_z(U)$ is Borel and integrable,
and it holds that
\begin{equation}\label{eq:MF}
\mu (U) = \int_Z \mu_z (U) \dd\nu(z).
\end{equation}
\end{enumerate}
\end{definition}


Note that if $\nu$ happens to be the push-forward of $\mu$ under the map $Y\to Z$ that sends each $Y_z$ to $z$ (when necessarily $\mu_z$ is a probability measure for $\nu$-a.e.\ $z\in Z$), then~\eqref{eq:MF} states that the map $z\mapsto \mu_z$ gives a disintegration of the measure~$\mu$, see e.g.~\cite[Section 5.1.2]{VianaOliveira:fet}. 

Under Definition~\ref{de:foliation}, the \emph{support} of a Borel set $U\subset \Omega$ is defined as
$$
\supp(U) :=\{z\in Z\colon \mu_z(U) >0\}.
$$ 
Of course, $\supp(U)=\supp(U\cap Y)$. Also, for $z\in Z$ and $X\subset Z$ we denote $U_z := U\cap Y_z$ and and $U_X :=\bigcup_{x\in X} U_{x}\subset Y$. Note that $U_z=U\cap Y_z$ is Borel by 
Item~\ref{it:foliation3} of the above definition. Also, since each $\mu_z$ is supported on $Y_z$,  we have $\mu_z(U) = \mu_z(U_z)$.

\begin{definition}\label{def:leaf-wise} Let $Y$ be a subset of $\Omega$ with a foliation $((Z,\C B_Z,\nu),(Y_z,\mu_z)_{z\in Z})$. For $\e>0$, we say that a finite set $S\subset \Ga$ is \emph{leaf-wise $\eps$-expanding} if, for every Borel $U\subset Y$, it holds that
$$
\mu_z (S.U)\ge \min\left((1-\eps)\,\mu_z(Y),\ \frac{\mu_z(U)}{\eps}\right),\quad\mbox{for every $z\in Z$}.
$$
 We say that $Y$ is a \emph{domain of leaf-wise expansion} if for every $\eps>0$ there is a leaf-wise
$\eps$-expanding finite set $S_\eps\subset \Ga$.

For $\D>0$, we say a finite set $T\subset \Ga$ is a \emph{$\D$-diffuser} for $Y$ if for every Borel $R\subset Z$ and every $z\in Z$ we have
\beq\label{eq-diffuser}
\mu_z (T.Y_R) \ge \D{\,}\mu(Y_R)\,/\,\nu(Z). 
\eeq
 A finite set  $T\subset \Ga$ is called a \emph{diffuser} if it is a $\D$-diffuser for some~$\D>0$.
\end{definition}

The following lemma states, informally speaking,  that a diffuser spreads well not only ring sets $Y_R$ 
but also those sets that are ``close'' to them.

\begin{lemma}\label{lm:diffuser}
Let $(Y_z,\mu_z)_{z\in Z}$ be a foliation of $Y\subset \Omega$.
Let $T$ be a diffuser for $Y$ and let $\D>0$ satisfy~\eqref{eq-diffuser}. Given $\eps\in (0,1)$, let $\de:=\eps\D/(\eps+|T|)$.
Then, for any Borel set $V\subset Y$ of positive measure satisfying $\mu_z(V)\ge (1-\de)\,\mu_z(Y)$ for each $z\in \supp(V)$, we have
 \beq\label{eq:admits-taking-transversals}
\nu\big(\left\{z\in Z\colon \mu_z(T.V) > \de{\,}\mu(V)/\nu(Z)\right\}\big) \ge (1-\eps)\,\nu(Z).
 \eeq
\end{lemma}
\begin{proof} Note that if we divide $\nu$ by some $\alpha>0$ and multiply each $\mu_z$ by the same constant $\alpha$,
then all statements of Definitions~\ref{de:foliation} and~\ref{def:leaf-wise} remain valid. (This will be used later in Remark~\ref{re:sample2} and this is why we divide by $\nu(Z)$ in~\eqref{eq-diffuser}.) Thus, we can assume without loss of generality that $\nu(Z)=1$.

Let $R:=\supp(V)$. Note that the set $R$ is Borel by Item~\ref{it:foliation4} of Definition~\ref{de:foliation}. The measure of $V_{Z\setminus R}$ is 0 by~\eqref{eq:MF}, so by removing this set from $V$  we may assume that $V_z=\emptyset$ for $z\in Z\setminus R$. Let $W:=Y_{R}\setminus V$.
Note that $Y_R = V\sqcup W$ and $\mu(Y_R)\neq 0$. Define
$$
R' := \left\{z\in Z\colon \mu_{z}(T.W) \ge \left(\D-\de\right)\mu(Y_R)\right\}.
$$ 
 Since $T$ is a $c$-diffuser,  we obtain that, for all
 $z \in Z\setminus R'$,
$$
\mu_z(T.V)=\mu_z(T.(Y_R\setminus W)) \ge \D{\,}\mu(Y_R) - \mu_z(T.W) > \de\,\mu(Y_R)\ge \de\,\mu(V).
$$

Thus, in order to finish the proof, it is enough to show that $\nu(R') \le\eps$.
 By~\eqref{eq:MF}, we have that 
 $$
 \mu(T.W)\ge \mu(T.W\cap Y_{R'})=\int_{R'}\mu_z(T.W)\dd\nu(z)\ge (\D-\delta)\,\mu(Y_R)\,\nu(R').
 $$
 Since $\mu_z(W)=\mu_z(Y\setminus V) \le\de{\,}\mu_z(Y)$ for each $z\in R$ by the assumption on $V$ and since $W\subset Y_R$, we have, again by \eqref{eq:MF}, that 
$
\mu(W)\le \de{\,}\mu(Y_R). 
$
Hence
$$
(\D-\de)\,\mu(Y_R)\,\nu(R')\le \mu(T.W)\le |T|\,\mu(W)\le |T|\,\de\,\mu(Y_R),
$$
giving the required bound $\nu(R')\le \eps$ by the choice of~$\delta$.
\end{proof}

We are ready to state our criterion for being a domain of expansion. Recall that $S_\eta$ for $\eta>0$ is  a finite leaf-wise $\eta$-expanding subset of~$\Ga$. Let us additionally assume that $S_\eta\ni e$.

\begin{prop}\label{prop:abstract-expansion}
Let $((Z,\C B_Z,\nu),(Y_z,\mu_z)_{z\in Z})$ be a foliation of Borel $Y\subset \Omega$. Suppose that there is $M\ge 1$  such that for all $z \in Z$ we have $M\ge\mu_z(Y)\ge\frac{1}{M}$.
If $Y$ is a domain of leaf-wise expansion which admits a diffuser $T$, then $Y$ is a domain of expansion.

More precisely, given  $\eta\in (0,1/2)$ and $\D>0$ such that $T$ is a $\D$-diffuser, let $\eps:=\eta/(2M^2)$, $\de:=\eps \D/(\eps+|T|)$, and $\be :=  \de\eps/(2M)$. Then  $R:=S_{\be}\, T \, S_\de\cup S_\de$ is an $\eta$-expanding set for $Y$.
\end{prop}

\begin{proof} 
By scaling the measures $\nu$ and $\mu$ by the same constant, we can assume for convenience that $\nu(Z)=1$.
By taking $R:=Z$ in~\eqref{eq-diffuser} and integrating it over all $z\in Z$, we see that $\D\le 1$. Thus
$\de\le \D \eps\le \eps\le 1/4$ and $\be\le 1/32$.

Take an arbitrary Borel set $U\subset Y$. We need to lower bound $\mu(R.U\cap Y)$. Define 
 $$
  X:=\{z\in Z\mid \mu_z(S_\de.U) \ge (1-\de)\,\mu_z(Y)\}.
 $$
 
\medskip\noindent \textbf{Case 1.} Suppose that $\mu(U_{Z\setminus X}) > \mu(U)/2$.

\medskip\noindent Since $S_\de$ is leaf-wise $\de$-expanding, we have for every $z\in Z\setminus X$ that $\mu_z(S_\de.U) \ge \mu_z(U)/\de$. Thus, by~\eqref{eq:MF}, 
we obtain the required lower bound as follows:
 $$
\mu(R. U\cap Y)  \ge \mu(S_\de.U\cap Y) 
\ge  \int_{Z\setminus X} \frac{\mu_z(U)}{\de}\dd\nu(z) = \frac{1}{\de}\, \mu(U_{Z\setminus X})\ge \frac1\delta\, \frac{\mu(U)}2 \ge \frac{\mu(U)}\eta.
 $$
The last inequality follows since, by $\D\le 1$, we have $\de\le \D\eps= \D\eta/(2M^2)\le \eta/2$.

\medskip\noindent \textbf{Case 2.} Suppose that Case 1 does not hold, that is, $\mu(U_X) \ge \mu(U)/2$.

Define $V:=(S_\de.U)_X \subset Y$.
Since $e\in S_\de$, we have $\mu(V) \ge \mu(U_X)\ge \mu(U)/2$. 
Let 
$$
 W:=\left\{z\in Z\colon \mu_z(T.V) >\de\, \mu(U)/2\right\}.
$$
 By Lemma~\ref{lm:diffuser} (that is, by our choice of $\delta$), we have
\beq\label{eq:nuW}
  \nu\left(W\right) > 1-\eps.
\eeq

We consider two subcases. First, suppose that $\mu(U) >\eps$.  For $z\in W$,  we have $\mu_z(T.V)> \delta\eps/2=\beta M\ge \beta \mu_z(Y)$ 
and thus $S_\beta$ cannot increase the $\mu_z$-measure of $(T.V)_z$ by factor $1/\beta$ or larger. Since
$S_\beta$ is leaf-wise $\beta$-expanding, it follows that  $\mu_z\left(S_{\beta} T.V \right)\ge (1-\beta)\,\mu_z(Y)$ for every $z\in W$. 
 Thus, by~\eqref{eq:MF} and~\eqref{eq:nuW},
 \begin{eqnarray*}
  \mu(R.U\cap Y) &\ge & 
\mu\left(S_{\beta} T.V\cap Y\right) 
 \ \ge\  
\int_W \left(1-\beta\right)\mu_z(Y)\dd\nu(z)\\
 &\ge&
\int_Z  \left(1-\beta\right)\mu_z(Y) \dd\nu(z)  \, -\, \eps M\\ 
 &\ge&
\left(1-\beta\right)\mu (Y) -  \eps M^2\mu(Y) \ \ge\ (1-\eta)\,\mu(Y).
\end{eqnarray*}
 as desired. (Note that $\beta+\eps M^2=\eps(\de/(2M)+M^2)=\eta(\delta/(4M^3)+1/2)\le \eta$.)

Finally, suppose that $\mu(U)\le \eps$. Let 
$$B:=\left\{z\in W\mid  \mu_z\left(S_{\beta} T.V \right)< (1-\beta)\,\mu_z(Y)\right\}.
$$
 For each $z\in B$, since $S_{\beta}$ is $\beta$-expanding for $Y_z$ and $B\subset W$, we have
$$
  \mu_z\left(S_{\beta} T.V \right)= \mu_z(S_\beta.(T.V)_z)\ge  \frac{\mu_z(T.V)}{\beta}\ge \frac{\de\, \mu(U)}{2\beta}.
$$
Thus, using~\eqref{eq:MF} again, we obtain
\begin{eqnarray*}
\mu(R.U\cap Y) &\ge & \mu\left(S_{\beta} T.V\cap Y\right) \ \ge\ \int_{W\setminus B}(1-\beta)\,\mu_z(Y) \dd\nu(z) + \int_{B}  \frac{\de\,\mu(U)}{2\beta}\dd\nu(z), \\
&\ge& \frac{1-\beta}M\,\nu(W\setminus B)+ \frac{\de\,\mu(U)}{2\be}\, \nu(B).
 \end{eqnarray*}
 We would like to argue that this is at least $\mu(U)/\eta$. Since both parts are linear in $\mu(U)$ it is enough to check only the extreme values of~$\mu(U)$. Suppose 
 that $\mu(U)=\e$ (as the inequality for $\mu(U)=0$ trivially holds by $\beta<1$). The coefficient at $\nu(B)$, namely $\de \e/(2\be)=M\ge 1$ is clearly larger than the coefficient at $\mu(W\setminus B)$. Thus (for $\mu(U)=\e$) the lower bound on $\mu(R.U\cap Y)$ is at least 
$$ 
 \frac{1-\beta}M\, \nu(W)\ge \frac{(1-\beta)(1-\eps)}M\ge \frac{31}{32}\cdot \frac34 \cdot \frac1{M}\ge \frac1{2M^2}=\frac{\mu(U)}{\eta},
$$
 as required.\end{proof}

We proceed to apply Proposition~\ref{prop:abstract-expansion}. Here, we have that $\Om = \R^{3}$ with the $L^2$-norm $\|\cdot\|$ and $\Ga = \iso(\R^{3})$. Define  \beq\label{eq:Y}
 Y:=\{y\in\R^3\colon 1 \le \|y\|\le \rho\},
 \eeq 
 where $\rho:= 1+\sqrt{2}/2$. For $d=1,2,3$, let $\mu^d$ denote the $d$-dimensional Hausdorff measure on $\I R^3$ scaled so that $\mu^d([0,1]^d\times\{0\}^{3-d})=1$. Let $\mu:=\mu^3$ be the Lebesgue measure on $\I R^3$.
The foliation of $Y$ is given by the concentric spheres, where $(Z,\nu)$ is the interval $[1,\rho]$ with the Lebesgue measure. Here, $\mu_z$ is the restriction of $\mu^2$ to $Y_z:=\{y\in\I R^3\mid \|y\|=z\}$, the sphere of radius~$z$.
By our scaling, $\mu_z(Y_z)=4\pi z^2$ for $z\in Z$. 
The map $Y\to (Z,\nu)\times (\I S^2,\mu^2)$, which sends $x$ to $(\|x\|, x/\|x\|)$, is a measure-preserving Borel isomorphism. Since a product of two standard finite measure spaces gives a foliation by Tonelli's theorem,
we conclude that we indeed have a foliation of~$Y$.

\begin{figure}[h]%
  \center{\resizebox{0.3\textwidth}{!}{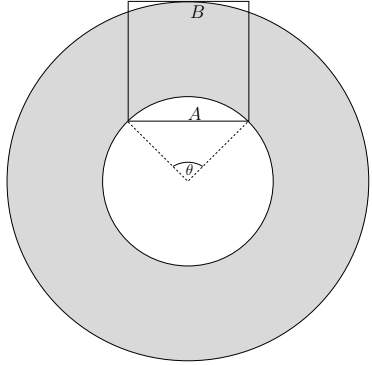}}
  \caption{\small Two-dimensional analogue of the set $Y$ (in grey) together with the cube $K$. The angle $\theta$ is the angle between two radii which pass through opposite vertices of the face $A$, and it is equal to ${\pi}/{2}$.}
  \label{fig-diffuser}
\end{figure}
The following lemma is an exercise in elementary geometry. 

\begin{lemma}\label{lem-cube} There is a solid 3-dimensional cube $K\subset \R^{3}$ with the following properties.
\begin{enumerate}[(i),nosep]
\item The side length of $K$ is $1$.
\item There are two opposing faces $A$ and $B$ of $K$ such that $A\cap Y_1$ consists of the four corners of $A$ while $B$ is tangent to $Y_\rho$.
\item For every $z\in Z$ and every point $x\in K\cap Y_z$ the angle between the plane tangent to $Y_z$ at $x$ and the plane extending the face $A$ is at most ${\pi}/{4}$.
\end{enumerate}
\end{lemma}
\begin{proof}[Sketch of Proof] The cube $K$ is pictured in Figure~\ref{fig-diffuser}. We construct it as follows. We start by inscribing into $Y_1$ a 2-dimensional square $A$ of side length $1$ 
whose third coordinate is a positive constant. Then we extend $A$ to a cube of side length $1$ in such a way that $A$ is its lower face (lower with respect to the third coordinate in $\R^3$).  The first property of the lemma is clear by construction.
Let $h$ be the height of a triangle whose vertices are the centre of $Y_1$ together with two opposite vertices of $A$.  One can easily compute that $h = {\sqrt{2}}/{2}$. Thus $h+1 =\rho$ and the second property follows.

Let us show the third property. By elementary considerations the angle in question is maximized at points $x$ which are corners of the face $A$ (when $z=1$). At such $x$ the angle in question is equal to the angle between the diagonal of $A$ and the tangent plane to $Y_1$, and that angle is easily computed to be equal to~${\pi}/{4}$.
\end{proof}

The following lemma, which gives $f\in \iso(\I R^3)$ such that $\{f\}$ is a diffuser, is a routine calculus exercise.

\begin{lemma}\label{lm:CubeDiffuser}
Let $f\colon K\to K$ be an isometry which maps the face $A$ to one of the side faces of~$K$. Let $R\subset [1,\rho]$ be a Borel subset. Then for all $z\in [1,\rho]$ we have
$$
\mu_{z}(f(K_R)) \ge \frac{1}{2}\,\mu(K_R)=\frac{1}{2\sqrt2}\,\frac{\mu(K_R)}{\nu(Z)}. 
$$
\end{lemma}

\begin{proof}
Let $g \colon K\to A$ be the  orthogonal projection onto $A$. Let $\cal L$ be the set of straight line segments contained in the face $A$ that connect $f(B)$ to $f(A)$ and are orthogonal to both $f(B)$ and $f(A)$.

Let us fix $z\in [1,\rho]$ and let $(g\circ f)_z$ be the restriction of $g\circ f$ to  $f^{-1}(K_z)$. Note that $(g\circ f)_z$ is a diffeomorphism from $f^{-1}(K_z)$ to $A$.  For $\ell\in \cal L$ let $\widehat \ell$ be the preimage of $\ell$ under $(g\circ f)_z$.

Let $U\subset \widehat\ell$ be a Borel subset. Note that $f$ maps $\widehat \ell$ isometrically into the sphere $K_z$. It follows that 
we have 
$$
\mu^1( (g \circ f)_z (U)) \ge \cos ({\pi}/{4})\, \mu^1(U) = \frac{\sqrt{2}}{2}\,\mu^1(U).
$$

Since the map $x\mapsto \|x\|$ does not increase distances and maps $K_R\cap \widehat\ell$ onto $R$, we have $ \mu^1(K_R\cap \widehat \ell) \ge \nu(R)$. Consequently we  have 
\beq\label{eq:ee}
\mu^1((g\circ f)_z (K_R\cap \widehat\ell)) \ge \frac{\sqrt{2}}{2}\, \nu(R).
\eeq

Let us fix $\ell\in \cal L$. We clearly have $A = (A\cap f(A))\times \ell$. Furthermore, the elements of $\cal L$ are exactly the sets $\{x\}\times \ell$ in this decomposition. Therefore, by Tonelli's theorem, and since the side length of $K$ is equal to $1$, 
$$
{\mu^2((g\circ f)_z(K_R))} \ge  \min_{\ell\in \cal L}\mu^1((g \circ f)_z(K_R)\cap \ell).
$$

Note that for $\ell \in \cal L$ we have $(g\circ f)_z(K_R)\cap \ell = (g\circ f)_z (K_R\cap \widehat\ell)$, and so by \eqref{eq:ee},
$$
    {\mu^2((g\circ f)_z(K_R))}  \ge \frac{\sqrt{2}}{2}\nu(R).
$$
Note that $(g\circ f)_z(K_R) = g(f(K_R)_z)$. Hence, since $g$ is an orthogonal projection, we have     $\mu_z(f(K_R)) = \mu^2(f(K_R)_z) \ge \mu^2((g\circ f)_z(K_R))$. On the other hand, by \eqref{eq:MF} we have $\mu(K_R) \le \nu(R)\sup\{\mu_z(K)\mid{z\in R}\}$.  Note that for $z\in R$ the set $K_z$ projects onto the face $A$. Hence, by the third item of Lemma~\ref{lem-cube}, we have that
$$ \sup\,\{\mu_z(K)\mid {z\in R}\} \le \frac{\mu^2(A)}{\cos(\frac{\pi}{4})} = \sqrt{2}.$$

Putting these inequalities together, we obtain that
$$
\mu_{z}(f(K_R)) \ge \mu^2((g\circ f)_z(K_R)) \ge \frac{\sqrt{2}}{2}\,\nu(R)\ge \frac{1}{2}\,\mu(K_R),
$$
which is exactly our claim.\end{proof}

\begin{corollary}\label{cr:astro-eta2} Let $\rho:= 1+\sqrt{2}/2$ and let $X\subset [1,\rho]$ be a Borel set of positive measure. Then the set $Y:=\{y\in\R^3: \|y\|\in X\}$ is a domain of expansion for the action $\iso(\I R^3)\actson \I R^3$. 
\end{corollary}
\begin{proof}  Recall that the set $Y$ comes with the spherical foliation by $Z=[1,\rho]$.
By the spectral gap property of $\SO(3)\actson \I S^3$ and Proposition~\ref{pr:Spectral},
the set $Y$ is a domain of leaf-wise expansion. By Lemma~\ref{lm:CubeDiffuser},  $Y$ admits a (single-element) diffuser. Now, the corollary follows from Proposition~\ref{prop:abstract-expansion}.
\end{proof}

We are ready to give another proof for the Euclidean case~of Theorem~\ref{th:DE}.

\begin{proof}[Second proof of Theorem~\ref{th:DE} for $\iso(\I R^n)\actson\I R^n$, $n\ge 3$.]
We use induction on $n$. If $n=3$, then $Y:= \{y\in\I R^3\mid 1 \le \|y\|\le \rho\}$ is a domain of expansion by Corollary~\ref{cr:astro-eta2}. Lemma~\ref{lm:Cover}  gives that every measurable subset $X\subset\I R^3$ that belongs to $\CU$  is a domain of expansion (since $X$ and $Y$ cover each other), as desired.

Suppose that $n\ge 4$. By Lemma~\ref{lm:Cover} again, we only need to show that $Y:=[0,1]^n$ is a domain of expansion.
Note that $Y$ becomes foliated when we take $Z:=[0,1]$ and $Y_z:=[0,1]^{n-1}\times\{z\}$ for $z\in Z$ (with the uniform probability measures).

Take any $\eta>0$. By the inductive assumption, there is a finite $\eta$-expanding set $S\subset \iso(\I R^{n-1})$ for $[0,1]^{n-1}$. Consider $S':=\{\ga'\mid \ga\in S\}\subset \iso(\I R^n)$, where we define $\ga'.x:=(\ga.(x_1,\dots,x_{n-1}),x_n)$ for $x\in\I R^n$ (that is, $\ga'$ acts at $\ga$ on the first $n-1$ coordinates and is trivial on the last coordinate).
The set $S'$ is leaf-wise-$\eta$-expanding. Thus
$Y$ is a domain of leaf-wise expansion.

Let $\ga$ be the isometry of $\R^{n}$ which is the identity on the first $n-2$ coordinates and which is the rotation by $\pi/2$ around the centre of $[0,1]^2$ on the last two coordinates. It is easy to see that for any Borel set $R\subset Z$ and any $z\in Z$ we have $\mu_z (\ga.Y_R) = \mu(Y_R)$. In particular, the singleton set $\{\ga\}$ is a diffuser (with the constant $\D=1$). Now, Proposition~\ref{prop:abstract-expansion} shows that $Y$ is a domain of expansion, finishing the proof.
\end{proof}


Also, it is easy now to derive Theorem~\ref{th:exotic} from the Introduction, namely  the existence of a domain of expansion $Y$ in $\I R^n$, $n\ge 3$, which is closed and nowhere dense. (Note that such a set $Y$ is meager and cannot essentially cover a non-empty open set.)

\begin{proof}[Proof of Theorem~\ref{th:exotic}.] Let $X$ be a closed nowhere dense subset of $[0,1+\sqrt2/2]$ of positive measure, for example, a ``fat'' Cantor set.

If $n=3$,  let $Y:=\{y\in\I R^3: \|y\|\in X\}$; otherwise let $Y:=[0,1]^{n-1}\times X$. The only non-trivial property that we have to check is that $Y$ is a domain of expansion. If $n=3$, this is the result of Corollary~\ref{cr:astro-eta2}. 

Suppose that $n\ge 4$. Here, we take the product-space foliation of $Y=[0,1]^{n-1}\times X$ by the last factor $Z:=X$. Since the action $\iso(\I R^{n-1})\actson \I R^{n-1}$ is expanding, $Y$ is a domain of leaf-wise expansion.
Similarly to the previous proof, one can show that $Y$ admits a diffuser consisting of a single isometry (which is trivial on the first $n-2$ coordinates and is a rotation by angle $\pi/2$ in the last two coordinates). Thus $Y$ is a domain of expansion by Proposition~\ref{prop:abstract-expansion}.\end{proof}

\section{Sample estimates of the number of pieces}\label{se:Numbers}

In this section we give sample estimates of  the number of pieces in some of our constructed equidecompositions. We will need the following result of Lubotzky, Phillips and Sarnak~\cite{LubotzkyPhillipsSarnak86cpam,LubotzkyPhillipsSarnak86cpam2}.

\begin{theorem}\label{th:LPS} For every integer $k\equiv 2\pmod 4$ with $k-1\ge 3$ being a prime, there is a symmetric $k$-set $Q\subseteq \SO(3)$ such that, for the action $\SO(3)\actson\I S^2$, the averaging operator $T_Q(f)=\frac1{k}\sum_{\de\in Q} \de.f$ on $L^2(\I S^2,\mu)$
has the second largest eigenvalue (in absolute value) at most $2\sqrt{k-1}/k$ and thus 
satisfies \beq\label{eq:TQ2}
   \|T_Qf\|_2\le (1-c)\|f\|_2,\quad\mbox{for every $f\in L^2(\I S^2,\mu)$ with $\int_{\I S^2} f\dd\mu=0$},
   \eeq
   where $c:=1-2\sqrt{k-1}/2k$.
   \end{theorem}   
   \bpf Let us outline just the construction of such a set. Take integer quaternions of the form $q=a_0+a_1 \mathbf{i}+a_2\mathbf{j}+a_3\mathbf{k}$ with $\sum_{i=0}^3 a_i^2=k-1$, $a_1,a_2,a_3$ even and $a_0>0$ odd, and then take the standard map that sends a unit quaternion $q/\sqrt{k-1}$ to an element of $\SO(3)$, see e.g.\ \cite[Equation (2.4)]{LubotzkyPhillipsSarnak86cpam2}. There are exactly $k$ such quaternions (\cite[Equation~(2.1)]{LubotzkyPhillipsSarnak86cpam2}) which come in conjugate pairs, giving a symmetric $k$-set of rotations. The stated spectral gap property is exactly the content of \cite[Theorem 2.1]{LubotzkyPhillipsSarnak86cpam2}.\epf

Note that the 
 Alon-Boppana bound~\cite{Nilli91} implies that no larger value of $c$ can work in~\eqref{eq:TQ2} for a symmetric $k$-set~$Q$.

First, we do one example where calculations are easy.

\begin{lemma}\label{lm:sample} Consider the action $\SO(3)\actson(\I S^2,\mu)$.
Let $A,B\subset \I S^2$  be measurable subsets of the same measure such that each contains a closed hemisphere. Then $A$ and $B$ are essentially Borel equidecomposable with at most $54$ pieces.
\end{lemma}
\begin{proof} Before equidecomposing, we can realign the sets so that they contain the same closed hemisphere $H\subset \I S^2$.
Fix an involution $\ga\in\SO(3)$ with $H\cup \ga.H=\I S^2$. Let  $T:=\{e,\ga\}$. Thus $T=T^{-1}$ and $T.A=T.B=\I S^2$. Let $Q\subset \SO(3)$ be the set returned by Theorem~\ref{th:LPS} for $k:=18$.
With $c:=1-2\sqrt{k-1}/k$ as in Theorem~\ref{th:LPS}, define
$$
f(x):=\frac{x}{(1-c)^2+2cx-c^2x}=\frac{81x}{64x+17},\quad \mbox{for $x\in [0,1]$}.
$$ 
 
Let us show that there is $\e>0$ such that 
 \begin{equation}\label{eq:FMin}
 f(x)\ge (1+\e)\min(2x,2/3),\quad \mbox{for every $x\in (0,1]$}.
 \end{equation}
 First, it is easy to see that  for all $x\in [0,1]$ we have $f(x)\ge \min(2x,2/3)$ (which reduces to a linear in $x$ inequality on each of the intervals $[0,1/3]$ and $[1/3,1]$) with equality if and only if $x=0$. 
  Also, $f'(0)=81/17>2$. Thus there is $\delta>0$ such that $f(x)\ge(1+\delta)\, 2x$ for all $x\in [0,\delta)$. By the continuity of the involved functions, there is $\e\in (0,\delta]$ such that~\eqref{eq:FMin} holds in the remaining compact interval $[\delta,1]$, thus proving~\eqref{eq:FMin}.

Fix $\e>0$ satisfying~\eqref{eq:FMin} and let $R:=TQ\cup (TQ)^{-1}$.

Let us show that the graph 
$G:=(A,B,E_{R}\cap (A\times B))$ is a bipartite $(\e/2)$-expander.
Take any measurable $U$ inside a part of~$G$. By $R=R^{-1}$ and the symmetry between $A$ and $B$, assume that  $U\subset A$. By~Proposition~\ref{pr:Spectral}\ref{it:Spectral1}, we have $\mu(Q.U)\ge f(u)$, where $u:=\mu(U)$.  (Recall that $\mu(\I S^2)=1$.) If $\mu(Q.U)\ge (2+\e)u$, then by $T.B=\I S^2$ we can choose $\be\in T$ with 
 \begin{equation}\label{eq:ChooseBeta}
 \mu(N(U))\ge \mu(TQ.U\cap B)\ge \mu(\be Q.U\cap B)=\mu(Q.U\cap \be.B)\ge \frac12\,\mu(Q.U),
 \end{equation}
 which is at least $
 (1+\e/2)u$,
 as desired. (Recall that $\be^{-1}=\be$ for every $\be\in T$.)
 Otherwise, by~\eqref{eq:FMin}, we have that $\mu(\I S^2\setminus Q.U)\le 1-f(u)<1/3$. Let $b:=\mu(B)$. If $b\ge 2/3$, then $\mu(B\setminus TQ.U)\le \mu(\I S^2\setminus Q.U) < b/2$. Otherwise we have  by the argument in~\eqref{eq:ChooseBeta} that
 $\mu(N(U))\ge\mu(Q.U)/2>b/2$.
 We conclude that $G$ is indeed a bipartite $(\e/2)$-expander.
 
By Theorem~\ref{th:LN}, $G$ admits an a.e.-perfect Borel matching. By Lemma~\ref{lm:MEQ}, this matching corresponds to an essential Borel equidecompositions between $A$ and $B$ using elements from $R$ only.
Since $Q$ is symmetric and $T\ni e$, we have that $|TQ\cap (TQ)^{-1}|\ge |Q|$ and thus $|R|\le 2\cdot |T|\cdot |Q|-|Q|=54$, as required.\end{proof}

Next, we give a rough estimate of the number of pieces needed for measurable equidecomposition between a cube and a Euclidean ball in $\I R^3$. 
Our purpose is just to demonstrate how an explicit bound can be derived from our results.

First, we need to understand how large the constructed $\eta$-expanding sets are.
For a domain of expansion $Y\subset \Om$ and $\eta\in (0,1)$, let $s(\eta;Y)$ be the smallest $s\in\I N$ such that there is a real $\de\in(0,\eta)$ and a $\de$-expanding set $S\subset \Ga$ for $Y$ with $|S|\le s$.

\begin{lemma}\label{lm:sSphere}
Let $\eta\in(0,1)$. Let $k\equiv 2\pmod 4$ be such that $k-1$ is a prime. Define $c:=1-2\sqrt{k-1}/k$ and 
$$
c':=c(2-c)=\frac{(k-2)^2}{k^2}.
$$
 If $\ell\in\I N$ satisfies $(1+c' \eta)^\ell>1/\eta$, then $s(\eta;{\I S^2})\le 1+\sum_{i=0}^{\ell-1} k(k-1)^i$.
 \end{lemma}
 \bpf Let $Q\subseteq \SO(3)$ be the symmetric $k$-set returned by Theorem~\ref{th:LPS} that satisfies~\eqref{eq:TQ2}.
By Proposition~\ref{pr:Spectral}\ref{it:Spectral2},
the set $Q^\ell$ is $\eta$-expanding for the whole space $\I S^2$. The size of $Q^\ell$ can be upper bounded by the number of words over $Q$ of length at most $\ell$ that do not contain an element next to its inverse, giving the stated bound.
\epf

\begin{lemma}\label{lm:sample2} Let $A$ and $B$ be, respectively, a cube and a ball  in $\R^3$ of the same volume. Let $Y$ be the annulus, as defined in~\eqref{eq:Y}. Then $A$ and $B$ are measurably equidecomposable with 
at most $1400\cdot s(1/1109;Y)+2000$ pieces.
\end{lemma}

\begin{proof} First we rescale $A$ and $B$ by the same factor so that each fits inside the annulus. Namely, we scale the ball to have radius $r:=(\rho-1)/2$. Then the cube has side length $\lambda:=(\frac{4\pi}3)^{1/3} r$ and elementary calculations show that it fits into $Y$ (e.g.\ touching the unit sphere $Y_1$ in the centre of its face). We can cover $Y$ by at most $m:=\lceil 2\rho/\lambda \rceil^3$ scaled cubes, arranged as a grid. Also, if we take a maximal packing of balls of radius $r/2$ with centres in $Y$, then balls with the same centres but of the twice larger radius $r$ cover $Y$. Since the smaller balls are disjoint and lie entirely between the spheres of radii $1-r/2$ and $\rho+r/2$, we have at most $n:=\lfloor((\rho+\frac r2)^3-(1-\frac r2)^3)/(\frac r2)^3\rfloor$ balls. Numerical calculations show that
$m=216$ and $n=1109$. 
Let $N,M\subset \iso(\I R^3)$ be the sets of translations of these sizes that achieve $M.A\supset Y$ and $N.B\supset Y$. 
By translating the ball and the cube, we can assume that each of $M$ and $N$ contains the identity.
Let  $S\subset \iso(\I R^3)$ be an $\eta$-expanding set for $Y$ with $\eta<1/1109$ of size $s(1/1109;Y)$.  Define $R:=S^{-1}M\cup N^{-1}S$. 

Let us show that the graph $G:=(A,B,E_R\cap (A\times B))$ generated by $R$ with parts $A$ and $B$ is a bipartite $c$-expander with $c:=1/(1109\,\eta)-1>0$. Indeed, take any measurable set $U$ inside one part, say $U\subseteq B$. Note that the neighbourhood of $U$ in $A$ is $R^{-1}.U\cap A\supset M^{-1}S.U\cap A$. If $\mu(S.U\cap Y)\ge \mu(U)/\eta$ then, by $M.A\supset Y$,
there is $\gamma\in M$ with 
$$
 \mu(\gamma^{-1}S.U\cap A)=\mu(S.U\cap \gamma.A)\ge \frac{\mu(S.U\cap Y)}m\ge \frac{\mu(U)}{\eta m}\ge (1+c)\mu(U),$$
 as desired. Otherwise, by the choice of $S$, we have $\mu(Y\setminus S.U)\le \eta\mu(Y)< \mu(A)/2$. (Note that $\mu(Y)/\mu(A)=(\rho^3-1)/r^3<1109/2$.) Since $M\ni e$ and $A\subset Y$, we have $\mu(M^{-1}S.U\cap A)\ge \mu(S.U\cap A)>\mu(A)/2$. Thus $G$ is indeed a bipartite $c$-expander.

By Theorem~\ref{th:LN}, the set $A$ and $B$ are essentially Borel equidecomposable using elements from $R$ only. 
In order to get a measurable exact decomposition between $A$ and $B$, it is enough by Proposition~\ref{pr:combine}\ref{it:combine1} to add some isometries to $R$ showing that $A$ and $B$ are set-theoretically equidecomposable. 
By the standard proof of Banach-Schr\"oder-Bernstein Lemma (Lemma~\ref{lm:BSB}), it is enough that the extra isometries show, in the notation of Section~\ref{se:aux}, that $[A]\preceq [B]$ and $[B]\preceq [A]$. The original proof of the Banach-Tarski Paradox shows that the ball doubling $[B]\sim 2[B]$ can be shown with 5 isometries and thus $2^k[B]\preceq [B]$ can be shown with $5^k$ isometries.
Rather roughly, 8 copies each of the sets $A$ and $B$ are enough to cover the other:
indeed, ``one eighth'' of each set, namely $[0,\lambda/2]^3$ and $\{v\in\I R^3\mid v\ge 0, \, \|v\|_2\le r\}$ respectively, can be covered by one translate of the other set. 
Thus $[A]\preceq 8[B]\preceq [B]$ can be shown with $8\cdot 5^3$ isometries. Also, $8[B]\preceq [B]\preceq 8[A]$ needs at most $5^3\cdot 8$ isometries. The relation $8[B]\preceq 8[A]$ means that these isometries can be used to construct a bipartite (multi-)graph $H$ with parts $A$ and $B$ such that every vertex of $A$ (resp.\ $B$) has degree at most $8$ (resp.\ exactly $8$). By Rado's theorem~\cite{Rado42}, $H$ has a matching covering every vertex of~$B$. Thus we can show $[B]\preceq [A]$ with at most 1000 isometries. In total, we need to add at most $2000$ further elements to $R$ to satisfy the lemma.\end{proof}

\begin{remark}\rm \label{re:sample2}
In order to estimate the number of pieces in Lemma~\ref{lm:sample2}, we need to estimate the function $s(\eta;Y)$ for the annulus~$Y$.
By Lemma~\ref{lm:CubeDiffuser}, there is a single-element $\D$-diffuser $T:=\{f\}$ with $\D:=1/(2\sqrt2)$. As we are free to scale each $\mu_z$ by $\alpha$ and $\nu$ by $1/\alpha$, we can assume that $\mu_1(Y)=1/\rho$ and $\mu_\rho(Y)=\rho$. Thus, when we apply Proposition~\ref{prop:abstract-expansion}, we can take $M:=\rho$. For any $\eta\in (0,1/2)$,  Proposition~\ref{prop:abstract-expansion} tells us to set $\eps:=\eta/(2\rho^2)$, $\delta:=\e\D/(\eps+|T|)$ and $\beta:=\delta\eps/(2\rho)$. In particular, if we take
$\eta=1/1109$, then our numerical calculations indicate that $\de=5.46...\cdot 10^{-5}$ and $\be=2.47...\cdot 10^{-9}$, and that $k:=6$ and $\ell:=4.3\cdot 10^5$ (resp.\ $\ell:=1.8\cdot 10^{10}$) satisfy  Lemma~\ref{lm:sSphere} when estimating $s(\de;Y)$ (resp.\ $s(\be;Y)$). 
Thus, by taking  minimum expanding sets $S_\de,S_\be\subset \SO(3)$ in Proposition~\ref{prop:abstract-expansion}, we get that at most 
$$1400\cdot |S_\de|\cdot(|S_\be|+1)+2000< 5^{1.81\cdot 10^{10}}<20^{10^{10}}
$$ 
pieces should be enough to measurably equidecompose the ball and the cube.
Of course, many improvements of this bound
are possible, even with 
very simple additional ideas.
However, we do not see a way to obtain any reasonable bounds here, so we do not pursue this direction any further.
\end{remark}

\section{Proof of Theorem~\ref{th:mean}}\label{se:mean}

Theorem~\ref{th:mean} follows from the following result where we have collected all the properties that are used in our proof: it is trivial to check that each action listed in Example~\ref{ex:1} (which is expanding by Theorem~\ref{th:DE}) 
 satisfies all assumptions of Lemma~\ref{lm:mean}.

\begin{lemma}\label{lm:mean} Under Assumptions~\ref{as:1} and~\ref{as:mu}, let the action $a:\Gamma\actson \Omega$ 
	be paradoxical and expanding. Suppose that the family of meager subsets of $\Om$ is $a$-invariant, the measure $\mu$ is atomless and there is a non-empty open set $W\subset \Om$ with $\mu(W)<\infty$. Let 
	$$\C A:=\{X\subset \Om\mid X\in \C B_\mu\cap\C T\mbox{ and $\O X$ is compact}\}.$$ 
		Then every 
	mean $\kappa:\C A\to [0,\infty)$ is a constant multiple of the measure~$\mu$.
\end{lemma}

\begin{proof} If $|\Om|=1$, then there is nothing to do. Otherwise, by e.g.\ the paradoxicality of the action, $\Om$ has infinitely many points. Fix pairwise disjoint non-empty open sets $U_1,U_2,U_3\subset\Om$ with compact closures.

	\begin{claim}\label{cl:MuKappa}
		Let $i\in [3]$ and $A,B\in\C A$ be subsets of $\Omega\setminus U_i$.
		\begin{enumerate}[(i),nosep]
			\item If $\mu(A)=\mu(B)$, then $\kappa(A)=\kappa(B)$. 
			\item If $\mu(A)\ge r\mu(B)>0$ for some $r\in\I R$, then $\kappa(A)\ge r\kappa(B)$.
		\end{enumerate}
	\end{claim}
	\begin{proof}[Proof of Claim.]
		Let us show the first part. The sets $A':=A\cup U_i$ and $B':=B\cup U_i$ are in $\C A$ (since $U_i\in\C A$ while $\C A$ is closed under unions). 
		
		Since the open set $W$ with $\mu(W)<\infty$ covers every other set in $\CU$ by Lemma~\ref{lm:minimal} and $\mu$ is invariant, we have 	for every $Y\in\CU$  that $\mu(Y)<\infty$ (and that $\mu(Y)>0$ since $Y$ covers a non-empty open set). In particular, $0<\mu(A')=\mu(B')<\infty$. Since
		the action $a$ is expanding, the sets $A'$ and $B'$ are domains of expansion. 
		By Proposition~\ref{pr:paradox}\ref{it:Paradox2} (resp.\ Theorem~\ref{th:suff}\ref{it:suff2}), the sets $A$ and $B$ are Baire (resp.\ Lebesgue) equidecomposable. Thus, by Proposition~\ref{pr:combine}\ref{it:combine2}, these sets are Baire-Lebesgue equidecomposable, say with pieces $A'=A_1\sqcup\ldots\sqcup A_m$ and $B'=B_1\sqcup\ldots\sqcup B_m$. Each piece belongs to $\C A$: it is in $\C B_\mu\cap\C T$ by definition and has compact closure as a subset of $A'$ or $B'$.
		By the finite
		additivity and $a$-invariance of $\kappa$, we have
		$$
		\kappa(A')=\sum_{j=1}^m \kappa(A_j)=\sum_{j=1}^m \kappa(B_j)=\kappa(B').
		$$
		Thus $\kappa(A)=\kappa(A')-\kappa(U_i)=\kappa(B')-\kappa(U_i)=\kappa(B)$, proving the first part of the claim.
		
		Let us show the second part. 
		
		First, assume that $r=p/q$ with $p,q\in\I N$.
		We use a consequence of Sierpi\'nski's theorem~\cite{Sierpinski22} (for a modern proof see e.g.\ \cite[Proposition~A.1]{DudleyNorvaisa11cfc}) that if $(\Omega',\C A',\mu')$ is an arbitrary non-atomic
		measure space with $\mu'(\Omega')<\infty$, then 
		for every $\rho\in [0,\mu'(\Omega')]$ there is $Y\in\C A'$ with $\mu'(Y)=\rho$. By applying this result to the $\sigma$-algebra
		$\C B_\mu\cap \C T\supset\C A$ of all Baire-Lebesgue subsets of $\Omega$,  we can find $(\C B_\mu\cap \C T)$-measurable partitions 
		$A=A_0\sqcup A_1\sqcup \ldots\sqcup A_p$ and $B=B_1\sqcup\ldots\sqcup B_q$ such that each of the $p+q$ sets $A_1,\ldots,A_p,B_1,\ldots,B_q$
		has measure exactly $\mu(B)/q$. Again, each piece is in~$\C A$. We conclude by the first part of the claim that all these sets, except $A_0$, have the same $\kappa$-value which has to be $\kappa(B)/q$.
		Thus
		$\kappa(A)\ge \sum_{i=1}^p \kappa(A_i)\ge p\,\kappa(B)/q$, as required. 
		
		Now, take an arbitrary positive $r\in\I R$. Suppose on the contrary that $\kappa(A)<r\,\kappa(B)$. Then, trivially, $\kappa(B)>0$ and there are $p,q\in\I N$ with $\kappa(A)< \frac pq\, \kappa(B)$ and $p/q<r$. However, this contradicts the conclusion of the previous paragraph.
	\end{proof}
	
	Observe that, for every $\mu$-null set $A\in\C A$, we have that $\kappa(A)=0$.
	Indeed, since $A= (A\setminus U_1)\sqcup (A\cap U_1)$, it is 
	enough by the additivity of $\kappa$ to prove this when $A\subset \Omega\setminus U_i$ for some $i\in[3]$. By taking $B=\emptyset$ in the first part of Claim~\ref{cl:MuKappa} we get that $\kappa(A)=\kappa(\emptyset)=0$, as desired. 
	
	This and the second part of Claim~\ref{cl:MuKappa}  give for every $i\in[3]$ that there is $c_i\in\I R$ such that, for every $\C A$-subset 
	$A$ of $\Omega\setminus U_i$, we have $\kappa(A)=c_i\,\mu(A)$. Furthermore, by $\kappa(U_3)=c_1\,\mu(U_3)=c_2\,\mu(U_3)$
	and $0<\mu(U_3)<\infty$, we conclude that $c_1=c_2$. This constant $c:=c_1=c_2$ works for every $A\in\C A$:
	$$\kappa(A)=\kappa(A\cap U_1)+\kappa(A\setminus U_1)=c_2\,\mu(A\cap U_1)+c_1\,\mu(A\setminus U_1)=c\,\mu(A),
	$$
	finishing the proof of the lemma.
\end{proof}

\section{Concluding Remarks}\label{se:conluding}

Note that there are paradoxical actions which are not expanding. For example, take two isometries in $SO(3)$ such that they generate a rank-2 free group $F$ and the
action $b:F\actson \I S^2$ has spectral gap, let $c:\I Z\actson \I S^1$ where the generator of $\I Z$ is a rotation by some angle $\theta\not\in \pi\I Q$, and consider the action $a:(F\times\I Z)\actson \Om$, where $\Om:=\I S^2\times \I S^1$ and $a((\gamma,n),(x,y)):=(b(\ga,x),c(n,y))$.
Since $b$ is paradoxical (e.g.\ by \cite[Theorem~5.5]{TomkowiczWagon:btp}), we can double the whole space $\Om$ also under the action $a$ (by using only elements of the form $(\gamma,0)\in F\times\I Z$).  Since every orbit of $b$ or $c$ is dense (in the former case by, for example, the proof of~\cite[Theorem~2.7]{LubotzkyPhillipsSarnak86cpam}), this also holds for the new action~$a$. Thus, with respect to the action $a$, every set with non-empty interior covers $\Om$ and is equidecomposable to $\Om$ by the proof of Proposition~\ref{pr:paradox}\ref{it:paradox1}. It follows that the action $a$ is paradoxical. On the other hand,  the products $\I S^2\times Y$, where we take for $Y$ 
almost $b$-invariant sets of measure $1/2$, are almost $a$-invariant sets. This shows that the action $a$ 
is not expanding.

An interesting open question that remains is whether any two bounded Borel subsets of $\I R^n$,
$n\ge3$, that have non-empty interior and the same Lebesgue measure are equidecomposable 
with Borel pieces. Our Theorem~\ref{th:MainRn} implies that the answer is in the affirmative,
provided we can first remove a Borel null set from each set.

Also, it would be interesting to find some alternative characterization of bounded measurable subsets of $\I R^n$, $n\ge 3$, that are domains of expansion. By  Lemma~\ref{lm:Cover}, essentially covering a non-empty open set is a sufficient condition. However, it is not necessary by Theorem~\ref{th:exotic}.

Of course, the above questions can also be asked for other group actions. In particular, they are open for all actions listed in Example~\ref{ex:1}, as far as the authors know.

Another interesting question is to get reasonable upper bounds on the sizes of expanding sets (namely to estimate the function $s(\eta;X)$ from Section~\ref{se:Numbers}) when, for example, $X=[0,1]^3$ is the unit cube under the action~$\iso(\I R^3)\actson \I R^3$.

\bigskip
\footnotesize
\noindent\textit{Acknowledgments.}
The authors thank Adrian Ioana, Greg Tomkowicz and P\'eter Varj\'u for useful discussions or comments. Also, the authors are grateful to the anonymous referee for helpful comments.

{\L}ukasz Grabowski was partially supported by EPSRC grant~EP/K012045/1, ERC Starting Grant
805495 and by Fondations Sciences Math{\'e}matiques de Paris during the program \textit{Marches Aléatoires et Géométrie Asymptotique des Groupes} at Institut Henri-Poincaré.

Andr\'as M\'ath\'e was partially supported by a Leverhulme Trust Early Career Fellowship
and by the Hungarian National Research, Development and Innovation Office -- NKFIH, grants no.\ 104178 and 124749.

Oleg Pikhurko was partially supported by ERC
grants~306493 and 101020255, EPSRC grant~EP/K012045/1 and Leverhulme Research Project Grant RPG-2018-424.


\end{document}

\end{document}